\documentclass[11pt]{amsart}

\usepackage{amsthm,amsmath,amssymb}
\usepackage[T1]{fontenc}
\usepackage[utf8]{inputenc}
\usepackage[english]{babel}
\usepackage{indentfirst}

\usepackage{graphicx}
\usepackage{math,math2}
\usepackage{nicefrac}
\newcommand{\calH}{\mathcal{H}}
\newcommand{\calK}{\mathcal{K}}
\newcommand{\calG}{\mathcal{G}}
\newcommand{\calC}{\mathcal{C}}
\newcommand{\calD}{\mathcal{D}}
\newcommand{\calM}{\mathcal{M}}
\newcommand{\calN}{\mathcal{N}}

\newcommand{\calJ}{\mathcal{J}}
\newcommand{\calA}{\mathcal{A}}
\newcommand{\calL}{\mathcal{L}}
\newcommand{\Ca}{\mathcal{C}(a)}
\newcommand{\Cinf}{C^{\infty}}
\newcommand{\omD}{\omega^{\mathcal{D}}}

\makeatletter
\newcommand{\neutralize}[1]{\expandafter\let\csname c@#1\endcsname\count@}
\makeatother

\newcommand{\dd}{\mathcal{\dagger}}
\newcommand{\ts}{\mathsf{T}}

\title[Hopf--Cole transform]{Hopf--Cole transformation via Generalized Schr{\"o}dinger bridge problem}
\date{\today}
\author[Léger]{Flavien Léger}
\address{UCLA, Department of mathematics, USA}
\email{flavien@math.ucla.edu}
\author[Li]{Wuchen Li}
\address{UCLA, Department of mathematics, USA}
\email{wcli@math.ucla.edu}
\thanks{This research is supported by 
AFOSR MURI FA9550-18-1-0502.}

\begin{document}
\begin{abstract}
We study generalized Hopf--Cole transformations motivated by the Schrödinger bridge problem, which can be seen as a boundary value Hamiltonian system on the Wasserstein space. We prove that generalized Hopf--Cole transformations are symplectic submersions in the Wasserstein symplectic geometry. Many examples, including a Hopf--Cole transformation for the shallow water equations, are given. Based on this transformation, energy splitting inequalities are provided.
\end{abstract}
\maketitle
\tableofcontents

\section{Introduction}
The Schrödinger bridge problem (SBP) nowadays plays a vital role in the physics, mathematics, engineering, information geometry and machine learning communities~\cite{ Carlen2014_stochastic,ChenGeorgiouPavon2016_relationa, Yasue1981_stochastica}. It was first introduced by Schr{\"o}dinger in 1931~\cite{schrodinger31,schrodinger32}, see also~\cite{fol88,Leonard2013_surveya}, and is closely related to, but different from the famous Schr{\"o}dinger equation. The SBP searches for the minimal kinetic energy density path for drift-diffusion processes with fixed initial and final distributions. In physics, Zambrini~\cite{Zambrini1986} related the SBP to the Schr{\"o}dinger equation derived in Nelson's stochastic mechanics~\cite{Carlen1984_conservative, Nelson1, Nelson2}. For numerical purposes, the SBP can be seen as an entropic regularization of optimal transport~\cite{CWZ2000,gs10}; its numerical solvers include the Sinkhorn algorithm \cite{Cuturi2013_sinkhornb} and Fisher information regularization method \cite{LiYinOsher2018_computations}. In information geometry and machine learning, the SBP has been studied as a statistical divergence function~\cite{LP}. In modeling, the SBP minimizing path, via Hopf--Cole transformation, shares similar structures with Nash equilibria in mean field games \cite{Lions}. 

Mathematically, the SBP can be viewed as a diffusion-relaxation of dynamical optimal transport \cite{Villani2003}. As a celebrated result~\cite{ Lafferty, Li2018_geometrya, Lott, Otto}, the $L^2$-Wasserstein metric introduces an infinite-dimensional Riemannian structure on the density space, therefore named \emph{density manifold}. In this aspect, the minimizing path of the SBP follows a boundary value Hamiltonian flow in density manifold. Recently, the study of SBPs as Hamiltonian flows has been developed in several works~\cite{Conforti2018_second, GentilLeonardRipani2018_dynamical}. For example, Conforti \cite{Conforti2018_second} proved several functional inequalities, connecting the value of the SBP to Ricci curvature lower bounds. One of the main techniques he used is the \emph{Hopf--Cole transformation}.

In this paper, we study a general family of SBPs, first considered in \cite{Leger2018}. This family of problems consists of controlled gradient flows of general potential energies on the density manifold. We introduce a generalized Hopf--Cole transformation for general potential energies \footnote{Later on, for simplicity we will omit the ``generalized'' part.}.  Following key insights from the Schr{\"o}dinger equation on graphs \cite{ChowLiZhou2017_discrete} and the Riemannian calculus in density manifold \cite{Li2018_geometrya}, we study the Wasserstein symplectic geometry for SBP. Our main results show that the Hopf--Cole change of variables is a symplectic embedding in the symplectic geometry of density manifold. We also prove several functional splitting inequalities related to our Hopf--Cole transformation. We demonstrate that some classical fluid dynamics, such as the shallow water equation, exhibit symplectic structures via Hopf--Cole transformations. In addition, the Hopf--Cole type transformations are also extended to general finite-dimensional manifolds with homogeneous of degree one type inverse metric tensors. 

The arrangement of this paper is as follows. In Section~\ref{section2}, we briefly review the generalized SBP and Wasserstein Riemannian calculus. In Section \ref{section3}, we study the Hopf--Cole transformation, which can be viewed as a symplectic change of variables. Many examples are provided therein. In Section \ref{section4}, we establish several energy splitting functional inequalities for these symplectic variables. Several extensions in finite-dimensional manifold as well as graphs are presented in Section \ref{section5}.

\section{Review}\label{section2}
 In this section, we briefly review the Schr{\"o}dinger bridge problem (SBP) and the Riemannian calculus in Wasserstein density manifold.

\subsection{Schrödinger bridge problem}
Consider a finite dimensional manifold $\big(M, (\cdot,\cdot)\big)$. For the simplicity of presentation, we shall assume $M$ to be compact and without boundary. Let $\nabla$ and $\div$ be the gradient and divergence operators on $M$, respectively. 

We first introduce the {\em dynamical SBP}, which takes the form
\begin{equation}
\label{eq:SB}
\inf_{\rho, b} \int_0^T \!\!\!\int \frac{1}{2} | b_t(x)|^2\,\rho_t(x)\,dx\,dt ,
\end{equation}
under dynamical constraints and fixed initial and final densities:
\[
\partial_t\rho_t + \div(\rho_t b_t) = \gamma\laplacian\rho_t,\quad\rho_0 = \mu, \quad \rho_T = \nu .
\]
Here $\gamma > 0$ is a diffusion parameter. Note that $\gamma=0$ corresponds to the Wasserstein-2 distance between $\mu$ and $\nu$ via the Benamou--Brenier formula~\cite{bb00}. The minimizer of problem~\eqref{eq:SB} has the property that $b_t=\nabla\Phi_t$ for some scalar field $\Phi_t$, and $(\rho, \Phi)$ satisfy a pair of partial differential equations---a Fokker--Planck equation and a Hamilton--Jacobi--Bellman equation:
\begin{equation}
\label{optimal}
\left\{\begin{aligned}
&\partial_t\rho_t+\textrm{div}(\rho_t \nabla\Phi_t)=\gamma \Delta\rho_t\\
&\partial_t\Phi_t+\frac{1}{2}\abs{\nabla\Phi_t}^2=-\gamma\Delta\Phi_t,
\end{aligned}
\right.
\end{equation}
for all $t\in (0,T)$. Although it is not apparent from formulation~\eqref{eq:SB}, the above system can be written in a time-symmetric fashion. To see this effect, consider the variable $S_t(x)=\Phi_t(x)-\gamma \log\rho_t(x)$. Then $(\rho, S)$ satisfies the system
\begin{equation}\label{optimal_S}
\left\{\begin{aligned}
&\partial_t\rho_t+\textrm{div}(\rho_t \nabla S_t)=0\\
&\partial_t S_t+\frac{1}{2}\abs{\nabla S_t}^2 = \delta\left(\frac{\gamma^2}{2} \int \abs{\nabla \log\rho_t}^2\rho_t \,dx \right), 
\end{aligned}
\right.
\end{equation}
together with the boundary conditions $\rho_0 = \mu$, $\rho_T = \nu$. Here $\delta$ denotes the first variation (i.e. $L^2$ gradient) with respect to $\rho$. Note that written in \eqref{optimal} or \eqref{optimal_S}, the Schrödinger bridge problem has the form of a first-order mean field game~\cite{Conforti2018_second, Lions}. More accurately, since the boundary conditions consist of fixing the densities, it is rather an instance of the \emph{planning problem}~\cite{LionsPlanning,Porretta2014,GMFT2018}. 

Rather formidably, system~\eqref{optimal_S} can be rewritten into a simpler and more symmetric way thanks to the {\em Hopf--Cole transformation}. Let
\begin{equation*}
{\eta_t(x)=\sqrt{\rho_t(x)}e^{S_t(x) / (2\gamma)},\quad \eta^*_t(x)=\sqrt{\rho_t(x)}e^{-S_t(x) / (2\gamma) },}
\end{equation*}
then $(\eta_t, \eta^*_t)$ satisfies a backward-forward heat system
\begin{equation}\label{bfheat}
\left\{\begin{aligned}
&\partial_t\eta_t=-\gamma\Delta \eta_t\\
&\partial_t\eta_t^*=\gamma\Delta\eta_t^* .
\end{aligned}\right.
\end{equation}
Integrating the above system in time leads to the so-called Schrödinger system, see~\cite{fortet1940,beurling1960,jamison1975}.

\subsection{Density manifold}
We next present the {\em Wassertein geometry} on the probability space, under which 
the SBP is a controlled gradient flow problem. In addition, the minimizing path of the SBP is a Hamiltonian flow in density space. Consider the set of smooth and strictly positive densities 
\begin{equation*}
\mathcal{P}_+(M)=\Big\{\rho \in C^{\infty}(M)\colon \rho(x)>0,~\int\rho(x)dx=1\Big\}. 
\end{equation*}
 \begin{definition}[$L^2$-Wasserstein metric tensor]\label{eqn:riemannian whole density space}
 Denote $$T_\rho\mathcal{P}_+(M) = \Big\{\dot\rho\in C^{\infty}(M)\colon \int\dot\rho(x) dx=0 \Big\}.$$ The $L^2$-Wasserstein metric $g^W_\rho\colon {T_\rho}\mathcal{P}_+(M)\times {T_\rho}\mathcal{P}_+(M)\rightarrow \mathbb{R}$ is defined by
 \begin{equation*}
 g^W_\rho(\dot\rho_1, \dot\rho_2)=\int \Big(\dot\rho_1(x), (-\Delta_\rho)^{\dd}\dot\rho_2(x)\Big) dx.
  \end{equation*}
Here $\dot\rho_1, \dot\rho_2\in T_\rho\mathcal{P}_+(M)$, $(\cdot, \cdot)$ is the metric on $M$ and
 $\Delta^{\dd}_\rho\colon {T_\rho}\mathcal{P}_+(M)\rightarrow{T_\rho}\mathcal{P}_+(M)$ is the inverse of the elliptical operator $$\Delta_\rho = \textrm{div}(\rho \nabla\cdot).$$
 \end{definition}
 
 Following \cite{Lafferty}, $(\mathcal{P}_+(M), g^W)$ is named \emph{density manifold}. Let us now briefly present the Riemannian calculus of $(\mathcal{P}_+(M), g^W)$, see related details in \cite{Li2018_geometrya}. 
 
 (i) The Christoffel symbol $\Gamma^W_\rho\colon T_\rho\mathcal{P}_+(M)\times\mathcal{P}_+(M)\rightarrow T_\rho\mathcal{P}_+(M)$ forms 
\begin{equation*}
\Gamma^W_\rho(\dot\rho_1, \dot\rho_2)=-\frac{1}{2}\Big\{\Delta_{\dot\rho_1}\Delta_{\rho}^{\dd}\dot\rho_2 +\Delta_{\dot\rho_2}\Delta_{\rho}^{\dd}\dot\rho_1+\Delta_\rho(\nabla\Delta_\rho^{\dd}\dot\rho_1, \nabla \Delta_\rho^{\dd}\dot\rho_2) \Big\}.
\end{equation*}
 
(ii) The Riemannian gradient forms  
 \begin{equation*}
\textrm{grad}_W\mathcal{F}(\rho)=\Big((-\Delta_\rho)^{\dd}\Big)^{\dd}\delta\mathcal{F}(\rho)
=-\Delta_\rho \delta\mathcal{F}(\rho)=-\textrm{div}(\rho \nabla  \delta\mathcal{F}(\rho)),
\end{equation*}
where $\mathcal{F}\colon \mathcal{P}_+(\Omega)\rightarrow \mathbb{R}$ and $\delta$ is the $L^2$ first-variation operator. 

(iii) The Riemannian Hessian operator $\textrm{Hess}_W\mathcal{F}(\rho)\colon T_{\rho}\mathcal{P}_+(M)\times T_{\rho}\mathcal{P}_+(M)\rightarrow \mathbb{R}$ forms as follows. Denote the cotangent vector of density manifold $\Phi\in C^{\infty}(M)/\mathbb{R}:=T_\rho^*\mathcal{P}(M)$, with the relation $V_\Phi=-\Delta_\rho \Phi\in T_\rho\mathcal{P}_+(M)$, then 
\begin{equation*}
\begin{split}
&\textrm{Hess}_W\mathcal{F}(\rho)(V_{\Phi}, V_{\Phi})\\=&\int\int \nabla_x\nabla_y \delta^2\mathcal{F}(\rho)(x,y)\nabla_x \Phi(x)\nabla_y \Phi(y) \rho(x)\rho(y)dx dy\\
&+\int \nabla_x^2 \delta \mathcal{F}(\rho) \abs{\nabla_x \Phi(x)}^2\rho(x)dx.
\end{split} 
\end{equation*}

\subsection{Hamiltonian flows}
A particular example for the Riemannian gradient is as follows. The gradient operator of the linear entropy (negative Shannon--Boltzmann entropy) forms the negative Laplacian operator. When $\mathcal{F}(\rho)=\int\rho\log\rho dx$, then
\begin{equation*}
\textrm{grad}_W\mathcal{F}(\rho)=-\Delta_\rho\delta\mathcal{F}(\rho)=-\textrm{div}(\rho \nabla (\log\rho +1))=-\Delta\rho.
\end{equation*}
From the above relation, the SBP can be viewed as a special case of the following variational problem:
\begin{equation}
\label{eq:GSB2}
\mathcal{A}(\mu, \nu)=\inf_{\rho_t, b_t}\int_0^T \!\!\!\int \frac{1}{2} |b_t(x)|^2\,\rho_t(x)\,dx,
\end{equation}
where $\rho_t,b_t$ are constrained by
\[
\partial_t\rho_t + \div(\rho_t b_t) = \div\big(\rho_t\nabla\delta\calF(\rho_t)\big),\quad 
\rho_0 = \mu, \quad \rho_T = \nu.
\]
We therefore call problem~\eqref{eq:GSB2} the generalized SBP (sometimes written GSBP)~\cite{Leger2018}, when $\calF$ is a general potential on the density manifold. This is the main focus of this paper.

We next demonstrate that SBP is a controlled gradient flow problem in density manifold. In details, we first present the Hodge decomposition at $\rho_t$, i.e. $$ b_t=\nabla\Phi_t+u_t,$$ where $u_t$ is the divergence free vector w.r.t. $\rho_t$, i.e. 
$\textrm{div}(\rho_t u)=0$. Thus
\begin{equation*}
\begin{split}
\int \abs{b_t(x)}^2\rho_t(x)dx=&\int \abs{\nabla\Phi_t(x)}^2\rho_t(x)+ \abs{u_t(x)}^2\rho_t(x)dx\\
\geq &\int \abs{\nabla\Phi_t(x)}^2\rho_t(x) dx.
\end{split}
\end{equation*}
Denote $a_t=-\Delta_{\rho_t}\Phi_t$. In this notation, notice that fact
\begin{equation*}
    \begin{split}
    g_{\rho_t}^W(a_t, a_t)=& \int_{M} \Big(-\Delta_{\rho_t}\Phi_t, (-\Delta_{\rho_t})^{\dd} (-\Delta_{\rho_t})\Phi_t\Big) dx\\
    =&\int (\Phi_t, (-\Delta_{\rho_t}\Phi_t)dx\\
    =&-\int \Phi_t\cdot\textrm{div}(\rho_t\nabla\Phi_t))dx\\
    =&\int\abs{\nabla\Phi_t}^2\rho_t dx.
\end{split}    
\end{equation*}
One can identity $a_t$ with $\Phi_t$ through the relation $a_t=-\Delta_{\rho_t}\Phi_t$. Thus the variational problem \eqref{eq:GSB2} is equivalent to the following formulation.
For all density paths $\rho_t\in \mathcal{P}_+(M)$ with boundary constraints $\rho_0=\mu$, $\rho_T=\nu$, consider 
 \begin{equation*}
\begin{split}
&\inf\Big\{\int_{0}^Tg^W_{\rho_t}(a_t,a_t)dt\colon\partial_t\rho_t=a_t-\textrm{grad}_W\mathcal{F}(\rho_t)\Big\}\\
=&\inf\Big\{\int_{0}^Tg^W_{\rho_t}(\partial_t\rho_t+\textrm{grad}_W\mathcal{F}(\rho_t),\partial_t\rho_t+\textrm{grad}_W\mathcal{F}(\rho_t))dt\Big\}\\
=&\inf\Big\{\int_{0}^T\Big(g^W_{\rho_t}(\partial_t\rho_t,\partial_t\rho_t)+g^W_{\rho_t}(\textrm{grad}_W\mathcal{F}(\rho_t), \textrm{grad}_W\mathcal{F}(\rho_t))\Big)dt\\
&\qquad+2\int_0^T(\textrm{grad}_W\mathcal{F}(\rho_t), \partial_t\rho_t) dt\Big\}\\
=&\inf\Big\{\int_{0}^Tg^W_{\rho_t}(\partial_t\rho_t,\partial_t\rho_t)+g^W_{\rho_t}(\textrm{grad}_W\mathcal{F}(\rho_t), \textrm{grad}_W\mathcal{F}(\rho_t)) dt\\
&\qquad+2(\mathcal{F}(\mu)-\mathcal{F}(\nu))\Big\},
\end{split}
\end{equation*}
where the first equality applies $a_t=\partial_t\rho_t- \textrm{grad}_W\mathcal{F}(\rho_t)$, the second equality expands the corresponding quadratic forms and the last equality uses the fact $$\int_0^Tg_\rho^W(\textrm{grad}_W\mathcal{F}(\rho_t), \partial_t\rho_t)dt=\int_0^T\frac{d}{dt}\mathcal{F}(\rho_t)dt=\mathcal{F}(\rho_t)|_{t=0}^{t=T}=\mathcal{F}(\mu)-\mathcal{F}(\nu).$$  
The above variational problem introduces a geometric action problem in density manifold. Its minimizing path satisfies an Euler--Lagrange equation in density manifold, i.e. the following second-order differential equation  
\begin{equation}\label{SOE}
\frac{D^2}{dt^2}\rho_t=\frac{1}{2}\textrm{grad}_Wg^W_{\rho_t}(\textrm{grad}_W\mathcal{F}(\rho_t), \textrm{grad}_W\mathcal{F}(\rho_t)),
\end{equation}
where $\frac{D^2}{dt^2}\rho_t=\partial_{tt}^2\rho_t+\Gamma_{\rho_t}^W(\partial_t\rho_t, \partial_t\rho_t)$ with $\partial_{tt}^2\rho=\frac{\partial^2}{\partial t^2}\rho(t,x)$ and $\Gamma^W_\rho$ is the Christoffel symbol in density manifold. In $L^2$-coordinates, the above equation is equivalent to 
\begin{equation*}
\begin{split}
&\partial_{tt}^2\rho_t-\Delta_{\partial_t\rho_t}\Delta_{\rho_t}^{\dd}\partial_t\rho_t-\frac{1}{2}\Delta_{\rho_t}(\nabla\Delta_{\rho_t}^{\dd}\partial_t\rho_t, \nabla\Delta_{\rho_t}^{\dd}\partial_t\rho_t)\\
=&-\frac{1}{2}\textrm{div}(\rho_t \nabla \delta\int |\nabla \delta\mathcal{F}(\rho_t)|^2\rho_tdx).
\end{split}
\end{equation*}

One can represent the above second order equations into two different first order systems. On the one hand, denote $\partial_t\rho_t=-\Delta_{\rho_t}(\Phi_t-\delta\mathcal{F}(\rho_t))$, then
\begin{equation}\label{optimal1}
\left\{\begin{aligned}
&\partial_t\rho_t=\partial_\Phi \mathcal{\bar H}(\rho_t, \Phi_t)\\
&\partial_t \Phi_t=-\partial_\rho \mathcal{\bar H}(\rho_t, \Phi_t),
\end{aligned}\right. 
\end{equation}
with the total Hamiltonian energy 
\begin{equation*}
\mathcal{\bar H}(\rho, \Phi)=\frac{1}{2}\int |\nabla\Phi(x)|^2\rho(x)dx-\int (\nabla \Phi(x), \nabla\delta \mathcal{F}(\rho))\rho(x)dx.
\end{equation*}
On the other hand, denote $\partial_t\rho_t=-\Delta_{\rho_t}S_t$, then 
\begin{equation}\label{optimal_S1}
\left\{\begin{aligned}
&\partial_t\rho_t=\partial_S \mathcal{H}(\rho_t, S_t)\\
&\partial_t S_t=-\partial_\rho \mathcal{H}(\rho_t, S_t),
\end{aligned}\right.
\end{equation}
with the total Hamiltonian energy 
\begin{equation*}
\mathcal{H}(\rho, S)=\frac{1}{2}\int \abs{\nabla S(x)}^2\rho(x)dx -\frac{1}{2}\int \abs{\nabla\delta \mathcal{F}(\rho)}^2\rho(x)dx.
\end{equation*}
Here the change of variable $\rho_t=\rho_t$, $S_t=\Phi_t-\delta \mathcal{F}(\rho_t)$ is canonical. And we simply check that in the classical case, when $\mathcal{F}$ is the linear entropy, the equation systems \eqref{optimal1}, \eqref{optimal_S1} are symplectic reformulations of the ones in \eqref{optimal}, \eqref{optimal_S} respectively.

It is clear that the minimizer path in GSBP is a Hamiltonian flow in density manifold. For the simplicity of presentation, we mainly focus on the formulation \eqref{optimal_S1}, which shows symmetry properties for later proofs. It wasn't explicitly mentioned in any previous work, which we are aware of, about the relation between Hopf--Cole transformation and Hamiltonian flow in density manifold. In this paper, we shall prove that the Hopf--Cole transformation $(\eta, \eta^*)$ performs the ``symplectic change of variables'' in density manifold. This fact is also true for general potential energy $\mathcal{F}(\rho)$.  

\section{Symplectic aspects of Hopf--Cole transformation}\label{section3}
In this section, we study the Hopf--Cole transformation and its generalization via the symplectic embedding in density manifold. 

We start by recalling the generalization of Hopf--Cole transformation for the GSBP introduced in~\cite{Leger2018}. Given $\dot\rho\in T_\rho\mathcal{P}_+(M)$, denote the cotangent vector $S=(-\Delta_\rho)^{\dd}\dot\rho\in T_\rho^*\mathcal{P}_+(M)$. Here $S$ is uniquely defined up to a spatial independent additive constant, i.e. $\nabla_x S(x)$ is uniquely determined. Thus we define the cotangent bundle of the density manifold as follows
\begin{equation*}
T^*\calP(M)=\Big\{(\rho, S)\colon \rho \in \mathcal{P}(M),~S\in C^{\infty}(M)/\mathbb{R}\Big\}.
\end{equation*}
Here the notation $C^{\infty}(M)/\mathbb{R}$ represents the set of smooth functions up to a spatial independent function. 

\begin{definition}[Hopf--Cole transformation for general potentials~\cite{Leger2018}]\label{GHC}
Given a potential functional $\calF\colon \mathcal{P}(M)\rightarrow \mathbb{R}$, the generalized Hopf--Cole transformation $$ s\colon C^{\infty}(M)\times C^{\infty}(M)\to T^*P(M),\quad (\eta,\eta^*)\to (\rho,S)$$ is given by 
\begin{equation*}
\left\{\begin{aligned}
{\nabla} \delta\mathcal{F}(\rho) = {\nabla}\Big(\delta\mathcal{F}(\eta) + \delta\mathcal{F}(\eta^*)\Big) \\
{\nabla} S = {\nabla}\Big( \delta\mathcal{F}(\eta) -\delta\mathcal{F}(\eta^*) \Big),
\end{aligned}\right.
\end{equation*}
where $\delta\calF$ denotes the first variation (i.e. $L^2$ gradient) of $\calF$. 
\end{definition}

In above definition, there are two spatial independent constants in the change of variable formula \eqref{eq:def-GHC}. For the simplicity of presentation, we let them to be zero. In other words, we consider 
\begin{equation}
\label{eq:def-GHC}
\left\{\begin{aligned}
 \delta\mathcal{F}(\rho) = \delta\mathcal{F}(\eta) + \delta\mathcal{F}(\eta^*)\\
 S = \delta\mathcal{F}(\eta) -\delta\mathcal{F}(\eta^*).
\end{aligned}\right.
\end{equation}

As a first example, consider the entropy $\mathcal{F}(\rho)=\gamma\int\rho\log\rho\,dx = \gamma\int\rho\,(\log\rho - 1)\,dx + \gamma$. In this case, $\delta\mathcal{F}(\rho)=\gamma\log\rho$, so that the generalized Hopf--Cole transformation 
\begin{equation*}
\left\{\begin{aligned}
\log\rho =& \log\eta + \log\eta^* \\
S/\gamma = &\log \eta - \log\eta^*
\end{aligned}\right.
\end{equation*}
corresponds to the classical Hopf--Cole transformation, i.e. $\eta=\sqrt{\rho}e^{S / (2\gamma)}$, $\eta^*=\sqrt{\rho}e^{-S / (2\gamma)}$.

We next discuss the relation between $(\eta, \eta^*)$ and $(\rho, S)$ via generalized Hopf--Cole transformation. Assume that $\delta\mathcal{F}$ is invertible. We set
\begin{equation*}
 \calC(M)=\Big\{(\eta,\eta^*)\colon
\begin{cases}
 \eta=\delta\mathcal{F}^{-1}\Big(\frac{1}{2}(\delta \mathcal{F}(\rho) + S)\Big),\\
 ~\eta^*=\delta\mathcal{F}^{-1}\Big(\frac{1}{2}(\delta \mathcal{F}(\rho) - S)\Big),
\end{cases}
(\rho, S)\in T^*\mathcal{P}(M)\Big\}.
\end{equation*}
We now state the main result of this section.

\begin{lemma}\label{lemma}
Given a solution $(\rho_t, S_t)$ to the Hamiltonian flow \eqref{optimal_S1}, the new variables $(\eta_t,\eta^*_t)$ satisfy
\begin{equation}\label{eq:HF-eta}
\left\{\begin{aligned}
\partial_t\eta_t = &\sigma(\eta_t,\eta^*_t)\, \partial_{\eta^*} \calK(\eta_t,\eta^*_t) \\
\partial_t\eta^*_t = &-\sigma(\eta_t^*,\eta_t)\, \partial_{\eta} \calK(\eta_t,\eta^*_t).
\end{aligned}\right.
\end{equation}
Here $\calK$ is the Hamiltonian in the new variables: $\calK(\eta,\eta^*) = \calH(\rho,S)$. Moreover $\sigma\colon C^{\infty}(M)\to C^{\infty}(M)$ is defined on test functions $\alpha, \beta\in C^{\infty}(M)$ by
\[
\bracket{\sigma(\eta,\eta^*) \alpha ,\beta} = \iint_{M\times M} \delta^2\calF(\rho)(x,y) \bar{\alpha}(x) \bar{\beta}(y)\,dx\,dy,
\]
where $\bar \alpha$ and $\bar \beta$ satisfy
\begin{align*}
\int \delta^2\mathcal{F}(\eta)(x,y)\bar{\alpha}(y)\,dy = \alpha(x), \\
\int \delta^2\mathcal{F}(\eta^*)(x,y)\bar{\beta}(y)\,dy = \beta(x).
\end{align*}
Equivalently, $\sigma(\eta,\eta^*)$ can be formulated as
\[
\sigma(\eta,\eta^*)(x, w) = -\frac{1}{2} \iint \big[\delta^2\calF(\eta)\big]^{-1}(x, y)\, \delta^2\calF(\rho)(y,z)\, \big[\delta^2\mathcal{F}(\eta^*)\big]^{-1}(z, w)\,dy\,dz .
\]
Note that $\sigma(\eta^*,\eta)$ is the adjoint of $\sigma(\eta,\eta^*)$, i.e. $\sigma(\eta^*,\eta)(x, y) = \sigma(\eta,\eta^*)(y, x)$. Here and throughout the text $\delta^2\calF$ denotes the second variation (i.e. $L^2$ Hessian) of $\calF$. 
\end{lemma}

\begin{remark}
Another way to interpret this result is by looking at the symplectic form in density manifold pulled back by $s$ on $\calC(M)$. For some special potentials $\calF$ it takes a particularly simple form, which doesn't depend on $(\eta,\eta^*)$. This symplectic perspective is developed in Section~\ref{sec:symplectic-forms} below.
\end{remark}

We delay the proof of this lemma and first present in the next two propositions important consequences for specific potentials of interest. We also refer to~Section~\ref{sec:examples} for more examples.

\begin{proposition}[Entropy-induced Hopf--Cole transformation]\label{prop:entropy}
Let $\calF(\rho) =\gamma \int\rho\log\rho \,dx$. Let $(\rho_t,S_t)$ be a solution to~\eqref{optimal_S1} and write $(\rho_t, S_t)=s(\eta_t,\eta^*_t)$ where $s$ is the Hopf--Cole transformation~\eqref{eq:def-GHC}. Then $(\eta_t, \eta^*_t)$ satisfies \eqref{bfheat}, which can be written as the Hamiltonian flow
 \begin{equation}
\left\{\begin{aligned}
\label{eq:HF-eta-1}
\partial_t\eta_t =& -(2\gamma)^{-1} \partial_{\eta^*} \calK(\eta_t,\eta^*_t)\\
\partial_t\eta^*_t =& (2\gamma)^{-1} \partial_{\eta} \calK(\eta_t,\eta^*_t),
\end{aligned}\right.
 \end{equation}
where $\calK$ is the Hamiltonian in the new variables: $\calK(\eta,\eta^*) = \calH(\rho,S)$. Specifically, in the entropy case the Hamiltonian in the new variables is given by $\calK(\eta,\eta^*) = -2\gamma^2\int\nabla \eta\cdot\nabla\eta^*dx$.
\end{proposition}

\begin{proposition}[Quadratic interaction energy induced Hopf--Cole transformation]
\label{th:quadratic}
Let $\calF(\rho) = \frac{1}{2}\iint W(x,y) \rho(x)\rho(y)\,dx\,dy$ with a symmetric positive-definite interaction potential function $W$. Let $(\rho_t,S_t)$ be a solution to~\eqref{optimal_S1} and write $(\rho_t, S_t)=s(\eta_t,\eta^*_t)$, where $s$ is the Hopf--Cole transformation~\eqref{eq:def-GHC}. Then $(\eta_t,\eta^*_t)$ satisfy the Hamiltonian flow
 \begin{equation}
\left\{\begin{aligned}
\label{eq:HF-eta2}
\partial_t\eta_t =& \sigma \,\partial_{\eta^*} \calK(\eta_t,\eta^*_t) \\
\partial_t\eta^*_t =& -\sigma\, \partial_{\eta} \calK(\eta_t,\eta^*_t) .
\end{aligned}\right.
 \end{equation}
Here $\calK$ is the Hamiltonian in the new variables: $\calK(\eta,\eta^*) = \calH(\rho,S)$. The constant linear map $\sigma\colon C^{\infty}(M)\to C^{\infty}(M)$ is defined by inverting the kernel $W$, 
\begin{equation}\label{eq:transform}
g = \sigma f \quad \iif \quad f(x) = \int W(x,y)g(y)\,dy,\quad\forall x\in M,
\end{equation}
for any test functions $f$ and $g$. 
\end{proposition}

Before proceeding with the proofs of Lemma~\ref{lemma} and Propositions~\ref{prop:entropy} and~\ref{th:quadratic}, let us emphasize that the interesting part of Propositions~\ref{prop:entropy} and~\ref{th:quadratic} is not so much that $(\eta_t,\eta^*_t)$ satisfy a Hamiltonian-type system of equations, but rather that the coefficient $\sigma$ doesn't depend on $(\eta,\eta^*)$. Indeed, symplectic theory says that for any Hamiltonian $\calH$, if $(\rho, S)$ satisfies the Hamiltonian flow equations~\eqref{optimal_S1}, then any smooth transformation $\tilde{s}\colon (\eta,\eta^*)\to (\rho, S)$ will yield a Hamiltonian flow-type equations in the new variables $(\eta,\eta^*)$. 
More precisely, for any smooth potential $\calF$, consider our Hopf--Cole transformation~\eqref{eq:def-GHC}. The following result shows that $(\eta,\eta^*)$ satisfy a Hamiltonian-type equation, as in Lemma \ref{lemma}. Note here that $\sigma$ depends on $(\eta,\eta^*)$ in general. The key point in Propositions~\ref{prop:entropy} and~\ref{th:quadratic} is that $\sigma$ is \emph{independent} of $(\eta,\eta^*)$. 
The only example of non-quadratic potentials $\calF$ for which $\sigma(\eta,\eta^*)$ is independent of  $(\eta,\eta^*)$ that we know of are the entropy with linear potentials, see the examples in Section~\ref{sec:examples} for additional details.

\begin{proof}[Proof of Lemma~\ref{lemma}]
We aim to derive an equation on $\partial_t\eta$, where $\eta$ is defined by the generalized Hopf--Cole transformation
\begin{equation}\label{eta}
\delta\mathcal{F}(\eta) = \frac{1}{2}\big(S + \delta\mathcal{F}(\rho)\big),
\end{equation}
and where $(\rho,S)$ satisfies the Hamiltonian system~\eqref{optimal_S1}, explicitely given by
\begin{equation*}
\left\{\begin{aligned}
&\partial_t\rho_t+\textrm{div}(\rho_t \nabla S_t)=0\\
&\partial_t S_t+\frac{1}{2}\abs{\nabla S_t}^2 = \delta\left(\frac{1}{2} \int \abs{\nabla \delta\calF(\rho_t)}^2\rho_t \,dx \right).
\end{aligned}
\right.
\end{equation*}
We start by taking time-derivatives on both sides of~\eqref{eta}:
\[
\int \delta^2\mathcal{F}(\eta_t)(x,y)\,\partial_t\eta_t(y) dy = \frac{1}{2}\partial_t S_t(x) + \frac{1}{2}\partial_t\delta\mathcal{F}(\rho_t)(x).
\]
We now compute each term in the R.H.S. of the above equation. On the one hand,
\begin{align*}
\partial_t S_t(x) &= -\frac{1}{2}\abs{\nabla S_t(x)}^2 + \frac{1}{2}\,\delta \Big(\int \abs{\nabla \delta\mathcal{F}(\rho_t)}^2\rho_t dx\Big)\\
 &= -\frac{1}{2}\abs{\nabla S_t(x)}^2 + \frac{1}{2}\abs{\nabla \delta\mathcal{F}(\rho_t)(x)}^2 +\int  \delta^2\mathcal{F}(\rho_t)(x,y) \Big(-\Delta_{\rho_t} \delta\mathcal{F}(\rho_t)(y)\Big)dy\\
&= -\frac{1}{2}\nabla\Big(S_t + \delta\mathcal{F}(\rho_t)\Big)(x) \cdot \nabla\Big(S_t - \delta\mathcal{F}(\rho_t)\Big)(x)\\
&\quad+\int  \delta^2\mathcal{F}(\rho_t)(x,y) \Big(-\Delta_{\rho_t} \delta\mathcal{F}(\rho_t)(y)\Big)dy, \\
\end{align*}
where we recall that $\laplacian_{\rho} = \div(\rho\nabla\cdot)$. Using the generalized Hopf--Cole transformation \eqref{eq:def-GHC} defining $\eta$ and $\eta^*$, we derive
\begin{equation}\label{1}
\partial_t S_t(x)=2 \nabla \delta\mathcal{F}(\eta_t)(x) \cdot \nabla\delta\mathcal{F}(\eta^*_t)(x) + \int  \delta^2\mathcal{F}(\rho_t)(x,y) \Big(-\Delta_{\rho_t} \delta\mathcal{F}(\rho_t)(y)\Big)dy.
\end{equation}
On the other hand, we check that 
\begin{equation}\label{2}
\begin{aligned}
\partial_t\delta\mathcal{F}(\rho_t)(x) &= \int \delta^2\mathcal{F}(\rho_t)(x,y)\partial_t\rho_t(y) dy \\
 &= \int \delta^2\mathcal{F}(\rho_t)(x,y)\Big(-\Delta_{\rho_t}S_t(y)\Big)dy.
\end{aligned}
\end{equation}
Combining \eqref{1} and \eqref{2}, we obtain 
\begin{equation*}
\begin{split}
&\partial_t \Big(S_t+\delta\mathcal{F}(\rho_t)\Big)(x)\\
=&2 \nabla \delta\mathcal{F}(\eta_t)(x)\cdot \nabla\delta\mathcal{F}(\eta^*_t)(x)+ \int  \delta^2\mathcal{F}(\rho_t)(x,y) (-\Delta_{\rho_t})\Big(S_t +  \delta\mathcal{F}(\rho_t)\Big)(y)dy. \\
\end{split}
\end{equation*}
The above equation can be simplified into
\begin{equation}\label{mid}
\partial_t\delta\mathcal{F}(\eta_t)(x) = \nabla\delta\mathcal{F}(\eta_t)(x) \cdot \nabla\delta\mathcal{F}(\eta_t^*)(x) + \int \delta^2\mathcal{F}(\rho_t)(x,y)(-\Delta_{\rho_t}) \big(\delta\mathcal{F}(\eta_t)\big)(y) \, dy.   
\end{equation}

Next, we would like to relate~\eqref{mid} to the Hamiltonian functional. Recall that the Hamiltonian is given by 
\begin{equation*}
\mathcal{H}(\rho, S) = \frac{1}{2} \int \Big(\abs{\nabla S(x)}^2  -\abs{\nabla\delta\mathcal{F}(\rho)(x)}^2\Big)\,\rho(x) \, dx.
\end{equation*}
Switching to $(\eta,\eta^*)$ variables, it is easy to check that the Hamiltonian can be written
\[
\mathcal{K}(\eta,\eta^*) = -2 \int \nabla\delta\mathcal{F}(\eta)(x) \cdot  \nabla\delta\mathcal{F}(\eta^*)(x)\, \rho(x)\,dx,
\]
where $\rho$ is understood as a function of $(\eta,\eta^*)$, i.e.

\begin{equation*}
\rho=(\delta\mathcal{F})^{-1}(\delta\mathcal{F}(\eta)+\delta\mathcal{F}(\eta^*)).
\end{equation*}
We now compute the first variation of $\mathcal{K}$ with respect to $\eta^*$
\begin{equation}\label{ah}
\begin{split}
\frac{\delta\calK}{\delta\eta^*(x)}=& -2 \left(\int  \nabla\delta\mathcal{F}(\eta)(y) \cdot \nabla\delta\mathcal{F}(\eta^*)(y) \,\frac{\delta \rho}{\delta\eta^*(x)}(y) \,dy\right.\\
& \left. +\int \nabla\delta\mathcal{F}(\eta)(y) \cdot \nabla_y\delta^2\mathcal{F}(\eta^*)(x,y) \,\rho(y) \,dy\right).
\end{split}
\end{equation}
Furthermore, by the relation between $\rho$ and $\eta$, $\eta^*$, we have
\[
\frac{\delta \rho}{\delta \eta^{*}(x)}(y) = \int 
[\delta^2\mathcal{F}(\rho)]^{-1}(y, z) \, \delta^2\mathcal{F}(\eta^*)(z, x) \,dz.
\]
Note that here the tensor $(\delta^2\mathcal{F})^{-1}$ denotes the inverse of the Hessian $\delta^2\mathcal{F}$. Applying it to \eqref{ah}, and inverting $\delta^2\calF(\eta^*)$ we derive 
\begin{equation*}
\begin{split}
&-\frac{1}{2} \int \big[ \delta^2\mathcal{F}(\eta^*) \big]^{-1}(z, x) \frac{\delta\calK}{\delta \eta^{*}(x)}\,dx  = \\ 
 & \int \nabla\delta\mathcal{F}(\eta)(y) \cdot \nabla\delta\mathcal{F}(\eta^*)(y) \big[ \delta^2\mathcal{F}(\rho) \big]^{-1}(y, z) \,dz - \Delta_{\rho} \big(\delta\mathcal{F}(\eta)\big)(z).
\end{split}
\end{equation*}
Applying the Hessian operator $\delta^2\calF(\rho)$ on both sides yields
\begin{equation*}
\begin{split}
&-\frac{1}{2} \iint \delta^2\calF(\rho)(y,z)\, [\delta^2\mathcal{F}(\eta^*)]^{-1}(z, x) \frac{\delta\calK}{\delta \eta^{*}(x)}\,dx\,dz  = \\ 
 &\nabla\delta\mathcal{F}(\eta)(y) \cdot \nabla\delta\mathcal{F}(\eta^*)(y) - \int \delta^2\calF(\rho)(y, z)\,\Delta_{\rho} \big(\delta\mathcal{F}(\eta)\big)(z)\,dz.
\end{split}
\end{equation*}
Comparing the above equation with~\eqref{mid}, we obtain 
\[
\partial_t\delta\mathcal{F}(\eta_t)(y)= -\frac{1}{2} \iint \delta^2\calF(\rho_t)(y,z)\, \big[\delta^2\mathcal{F}(\eta_t^*)\big]^{-1}(z, x) \frac{\delta\calK}{\delta \eta^{*}(x)}\,dx\,dz .
\]
Finally, note that 
\begin{equation*}
\partial_t\delta\mathcal{F}(\eta_t)(y) = \int \delta^2\mathcal{F}(\eta_t)(y, x) \partial_t\eta_t(x) \,dx.
\end{equation*}
By applying to both sides $\delta^2 \mathcal{F} (\eta_t)^{-1}$, we derive 
\[
\begin{split}
\partial_t\eta_t(x) = -\frac{1}{2} \iiint &\big[\delta^2\calF(\eta_t)\big]^{-1}(x, y)\, \\
&\delta^2\calF(\rho_t)(y,z)\, \\
&\big[\delta^2\mathcal{F}(\eta_t^*)\big]^{-1}(z, w) \frac{\delta\calK}{\delta \eta^{*}(w)}\,dy\,dz\,dw,
\end{split}
\]
which is precisely the first equation in \eqref{eq:HF-eta}. 
By very similar arguments, we can obtain the second equation for $\partial_t\eta^*_t$ in \eqref{eq:HF-eta}. 
\end{proof}

Here the above derivation is computed directly. For readers who are not familiar with $L^2$ variations, we refer to Lemma~\ref{lemma:manifolds:HF-eta}, in which a finite-dimensional analogue of the proof is given. In addition, Proposition~\ref{prop:entropy} and~\ref{th:quadratic} will be proved later on in the examples of Section~\ref{sec:examples}.

\subsection{Symplectic forms}
\label{sec:symplectic-forms}

In this section we interpret Lemma \ref{lemma} in the framework of symplectic geometry. In the case of the Schr{\"o}dinger equation and the associated Madelung transform, the symplectic perspective has been studied in~\cite{Lafferty, vonRenesse2008_optimal}. In this section, we study the similar relation in Hopf--Cole transformation and its generalization. We begin our exposition by quickly recalling basic concepts such as symplectic submersions.
\begin{definition}
Let $(\calM, \om)$ and $(\calN, \eta)$ be two symplectic manifolds. A \emph{symplectomorphism} is a diffeomorphism $s\colon (\calM, \om)\to (\calN, \eta)$, which satisfies
\[
\eta(ds(X),ds(Y)) = \om(X,Y)
\]
for all vector fields $X, Y\in T\calM$.
\end{definition}

Symplectomorphisms have desirable properties. For instance they preserve Hamiltonian flows:

\begin{proposition}
Let $s\colon (\calM, \om)\to (\calN, \eta)$ be a symplectomorphism. Let $\calH\in \Cinf(\calN)$, $\calK\in \Cinf(\calM)$ such that $\calH = \calK\circ s$. Then $s$ maps Hamiltonian flows associated to $\calH$ on $(\calN, \eta)$ to Hamiltonian flows associated to $\calK$ on $(\calM,\om)$.
\end{proposition}

As a consequence, given a symplectic manifold $(\calN,\eta)$, a manifold $\calM$ and a diffeomorphism $s\colon \calM\to \calN$, one can \emph{define} a symplectic form on $\calM$ by $\om(X,Y) = \eta(ds(X),ds(Y))$ which turns $s$ into a symplectomorphism. We are interested in the cases where such $\om$ is canonical for $\calM$, or more generally a constant symplectic form (i.e. constant in some appropriate coordinates). 

In our setting, the vector space $\calC(M)=\Cinf(M)\times \Cinf(M)$ plays the role of the manifold $\calM$ above, the cotangent bundle to the Wasserstein space $T^*\calP(M)$
plays the role of $\calN$, and we equip it with the natural symplectic form associated to the Wasserstein metric.
Given $(\dot\rho_1, \dot S_1), (\dot\rho_2, \dot S_2)\in T_{(\rho, S)}T^*\mathcal{P}_+(M)$, where
$$T_{(\rho, S)}T^*\mathcal{P}_+(M)=\{(\dot\rho, \dot S)\colon \int \dot\rho (x) dx=0,~\dot S\in C^{\infty}(M)/\mathbb{R}\},$$ the symplectic structure in density manifold $\omega_\rho^{\calW}\colon TT^*\mathcal{P}_+(M)\times TT^*\mathcal{P}_+(M)\rightarrow\mathbb{R}$ satisfies 
\begin{equation*}
\omega^{W}_\rho((\dot \rho_1, \dot S_1),(\dot\rho_2, \dot S_2))= \int \dot\rho_1(x)\dot S_2(x)- \dot \rho_2(x)\dot S_1(x)dx.
\end{equation*}

In the case where the map $s$ is our Hopf--Cole transformation given by~\eqref{eq:def-GHC}, we can state a different version of Proposition~\ref{th:quadratic}:

\begin{theorem}
\label{th:symplectic}
Define the symplectic form $\omega^{\calC}_\rho$ on $\calC(M)$ by
\begin{equation*}
\omega^{\calC}_\rho\Big((\dot\eta_1, \dot \eta^*_1), (\dot\eta_2, {\dot \eta}^*_2) )\Big)=\int\bracket{\sigma(\eta,\eta^*) \dot\eta_1 ,\dot\eta^*_2}-  \bracket{\sigma(\eta,\eta^*) \dot\eta_2 ,\dot\eta^*_1} dx,
\end{equation*}

where $(\eta,\eta^*)=s^{-1}(\rho, S)$ and 
\begin{equation*}
\rho=(\delta\mathcal{F})^{-1}(\delta\mathcal{F}(\eta_1)+\delta\mathcal{F}(\eta_1^*))=(\delta\mathcal{F})^{-1}(\delta\mathcal{F}(\eta_2)+\delta\mathcal{F}(\eta_2^*)).
\end{equation*}
Then the Hopf--Cole transform~\eqref{eq:def-GHC} is a symplectomorphism between $(\calC(M), \omega^{\calC}_\rho)$ and $(T^*\calP(M),\omega^{W}_\rho)$ .
\end{theorem}

\begin{proof}
We simply write the Lemma \ref{lemma} into the following symplectic form $\omega^{\calC}_\rho$, i.e. 
\begin{equation*}
\begin{split}
\omega^{\calC}_\rho\Big(((\dot\eta_1, \dot \eta^*_1),(\dot\eta_2, {\dot \eta}^*_2) )\Big)= \omega^{W}_\rho \Big((\dot\rho_1, \dot S_1), (\dot\rho_2,\dot S_2)\Big).
\end{split}
\end{equation*}

\end{proof}
\begin{remark}
Our result is based on the cotangent vectors in density manifold, developed by \cite{Li2018_geometrya}. Compared to cotangent vector fields used in manifold $M$ \cite{Lott,Lafferty, vonRenesse2008_optimal}, our proofs are direct. In addition, our results of symplectic forms hold for general potential energies. 
\end{remark}
\begin{remark}
Like previously remarked, when $\mathcal{F}$ are linear entropy or quadratic interaction energy, we emphasize that $\omega^{\calC}_\rho$ is independent of $(\eta,\eta^*)$ in $L^2$ coordinates. This makes Theorem~\ref{th:symplectic} attractive. 
\end{remark}


\subsection{Examples} \label{sec:examples}
We now examine a list of important examples, for which we express in more details our Hopf--Cole transformation and the equations satisfied by the new variables.

\begin{example}[Rényi entropy induced Hopf--Cole]\label{example1}
Let $m>0$, $\gamma\in \mathbb{C}$ and consider the so-called \emph{Rényi entropy}
\[
\calF(\rho) = \frac{\gamma}{m(m+1)}\int\rho(x)^{m+1}dx.
\]
First, the Hamiltonian flow equations~\eqref{optimal_S1} take the form

\begin{equation}
\label{eq:HF-Renyi-entropy}
\begin{cases}
\partial_t\rho_t + \div(\rho_t\nabla S_t) = 0 \\
\partial_t S_t + \frac{1}{2}\abs{\nabla S_t}^2 = \gamma^2 \Big[ (m-\frac{1}{2})\rho_t^{2m-2}\,\abs{\nabla\rho_t}^2 - \div(\rho_t^{2m-1}\nabla\rho_t)\Big] .
\end{cases}
\end{equation}
Since $\calF(\rho) = \frac{\gamma}{m(m+1)}[\int\rho(x)^{m+1}-(m+1)\rho \,dx + (m+1)]$, its first variation can be written $\delta\calF(\rho) =  \frac{\gamma}{m} (\rho^m-1)$, and our Hopf--Cole transformation takes the form $\rho^m+1 = \eta^m + \eta^{*m}$, $\frac{m}{\gamma} S = \eta^m - \eta^{*m}$, i.e.
\begin{equation*}
\left\{\begin{aligned}
\eta(x) &= 2^{-1/m} \,\,\big(\rho(x)^m + mS(x)/\gamma + 1\big)^{1/m} \\
\eta^*(x) &= 2^{-1/m} \,\,\big(\rho(x)^m - mS(x)/\gamma + 1\big)^{1/m} .
\end{aligned}\right.
\end{equation*}
The two ingredients needed to express the Hamiltonian flow in $(\eta,\eta^*)$ variables are the form of the Hamiltonian $\calK(\eta,\eta^*)$ as well as the linear map $\sigma(\eta,\eta^*)$ associated to the symplectic form $\omD$. It is a simple matter to check that the Hamiltonian is given by
\[
\calK(\eta,\eta^*) = -2(\gamma/m)^2 \int \nabla (\eta^m)\cdot \nabla (\eta^{*m})\,(\eta^m + \eta^{*m}-1)^{1/m}dx.
\]
The map $\sigma(\eta,\eta^*)$ is given by the diagonal kernel
\[
k(x)=(2\gamma)^{-1} \big(\eta(x)^{-m} + \eta^*(x)^{-m} - \eta(x)^{-m}\eta^*(x)^{-m}\big)^{\frac{m-1}{m}},
\]
by which we mean that $\bracket{\sigma(\eta,\eta^*) f, g} = \int k(x) f(x) g(x)\,dx$ for any test functions $f$ and $g$. Consequently the equations satisfied by $(\eta,\eta^*)$ are 
\begin{equation}\label{CV}
\left\{
\begin{aligned}
\partial_t\eta_t &= \gamma m^{-1} \,\nabla\eta_t^{m}\cdot\nabla\eta_t + \gamma m^{-1}\, \eta_t^{-(m-1)}(\eta_t^m + \eta_t^{*m}-1)\laplacian\eta_t^m \\
\partial_t\eta^*_t &= -\gamma m^{-1}\, \nabla\eta_t^{*m}\cdot\nabla\eta^*_t - \gamma m^{-1}\, \eta_t^{*-(m-1)} (\eta_t^m + \eta_t^{*m}-1)\laplacian\eta_t^{*m} .
\end{aligned}\right.
\end{equation}
\end{example}

\begin{example}[Hopf--Cole for linear entropy]
We next examine the Hopf--Cole formula for linear entropy.
Note the fact that as $m\to 0$, $z^m = 1 + m\log z+ o(m)$. Thus
\begin{equation*}
\lim_{m\rightarrow 0}\frac{1}{m(m+1)}\int\rho(x)^{m+1} -\rho(x)dx=\int\rho(x)\log\rho(x) dx. 
\end{equation*}
Again notice that $\eta^m=1+m\log\eta+o(m)$. Then 
\begin{equation*}
\begin{split}
&\lim_{m\rightarrow 0} \Big(m^{-1} \,\nabla\eta^{m}\cdot\nabla\eta + m^{-1}\, \eta^{-(m-1)}(\eta^m + \eta^{*m} - 1)\laplacian\eta^m\Big)\\
=& \nabla\log\eta\cdot \nabla\eta + \eta\Delta\log\eta.
\end{split}
\end{equation*}
Recall that the Laplacian operator has the following format:
\begin{equation*}
\Delta\eta=\nabla\cdot(\eta \nabla \log\eta)=\nabla\eta \cdot\nabla \log\eta +\eta \Delta \log\eta,
\end{equation*}
therefore the classical Hopf--Cole transform is recovered when $m\to 0$. In other words, substituting the above two relations into equation \eqref{CV}, we derive the backward-forward heat system \eqref{bfheat}.

\end{example}

\begin{example}[Hopf--Cole for shallow water equations]
Consider the interaction kernel $W(x, y) = \gamma \,\delta_{x=y}$, with $\gamma > 0$, which corresponds to the potential
\[
\calF(\rho) = \frac{\gamma}{2}\int \big(\rho(x)\big)^2\,dx .
\]
This functional is related to the \emph{shallow water equations}, see~\cite{hamilton1982}. Note also that it is a particular case of the previous example: a Rényi entropy with $m=1$. In this situation the map $\sigma$ is especially simple since it is (a multiple of) the identity: $\bracket{\sigma f,g} = (2\gamma)^{-1} \int f(x) g(x)\,dx$. Therefore, the new Hamilton equations in $(\eta,\eta^*)$ variables are
\begin{equation*}
\left\{\begin{aligned}
\partial_t\eta_t =& (2\gamma)^{-1}\, \partial_{\eta^*} \calK(\eta_t, \eta^*_t) \\
\partial_t\eta^*_t =& - (2\gamma)^{-1}\, \partial_{\eta} \calK(\eta_t, \eta^*_t) .
\end{aligned}\right.
\end{equation*}
Here $\calK(\eta,\eta^*)=\frac{\gamma}{2}\int \big(\eta(x)+\eta^*(x)\big)\, \nabla \eta(x)\cdot\nabla\eta^*(x)\,dx$. In symplectic terms, we can say that $\omD$ is (a multiple of) the canonical symplectic form on $\calD(M)$. 
More specifically, the previous system of equation is
\begin{equation*}
\begin{cases}
\gamma^{-1}\,\partial_t\eta_t + \frac{1}{2}\abs{\nabla\eta_t}^2 + (\eta_t + \eta^*_t)\laplacian\eta_t = 0 \\
-\gamma^{-1}\,\partial_t\eta^*_t + \frac{1}{2}\abs{\nabla\eta^*_t}^2 + (\eta_t + \eta^*_t)\laplacian\eta^*_t = 0. 
\end{cases}
\end{equation*}
\end{example}

\begin{example}[Interaction energy]
Let $W\colon \Rd\to\mathbb{R}$ be an even interaction kernel, i.e. $W(z)=W(-z)$, for $z\in\mathbb{R}^d$, and consider the potential
\[
\calF(\rho)=\frac{1}{2}\iint_{\Rd\times \Rd} W(x-y)\rho(x)\rho(y)\,dx\,dy .
\]
The Hamiltonian flow~\eqref{optimal_S1} associated to $\calF$ is
\begin{equation*}
\left\{\begin{aligned}
&\partial_t\rho_t + \div(\rho_t\nabla S_t) = 0 \\
&\partial_tS_t + \frac{1}{2}\abs{\nabla S_t}^2 = \frac{1}{2}\abs{\nabla W\conv \rho_t}^2 - \div(\rho_t\nabla W\conv \rho_t)\conv W ,
\end{aligned}\right.
\end{equation*}
where $\conv$ denotes convolution. We assume that convolution with $W$ is invertible, i.e. there exists another kernel $M\colon\Rd\to\Real$ such that 
\[
M\conv W\conv \zeta= \zeta 
\]
for all smooth function $\zeta\colon\Rd\to\Real$. Then, our Hopf--Cole transformation is well-defined, $W\conv\rho = W\conv\eta + W\conv\eta^*$, $S = W\conv\eta - W\conv\eta^*$, i.e. 
\begin{equation*}
\left\{\begin{aligned}
\eta &= \frac{\rho + M\conv S}{2} \\
\eta^* &= \frac{\rho - M\conv S}{2} .
\end{aligned}\right.
\end{equation*}
\end{example}

\subsection{Extensions to Madelung transformation}
Proposition \ref{prop:entropy} introduces the connection between the Schrödinger bridge problem and the Schr{\"o}dinger equation. When $\gamma=\sqrt{-1}$ is the imaginary unit and $\mathcal{F}(\rho)=\gamma\int \rho\log\rho dx$, the generalized Hopf--Cole transformation
\begin{equation*}
\left\{\begin{aligned}
\log\rho =& \log\eta + \log\eta^* \\
S/\sqrt{-1} = &\log \eta - \log\eta^*, 
\end{aligned}\right.
\end{equation*}
forms exactly the Madelung transformation, i.e. $\eta=\sqrt{\rho}e^{-\sqrt{-1}S/2}$, $\eta^*=\sqrt{\rho}e^{\sqrt{-1}S/2}$. And it is clear that $\eta$ is the complex conjugate of $\eta^*$. As known already in \cite{Nelson1, Lafferty}, \begin{equation*}
-\sqrt{-1}\partial_t\eta_t=\frac{1}{2}\Delta\eta_t    
\end{equation*}
satisfies the Schr{\"o}dinger equation. 

In this sequel, we further consider the generalized potential energies for Madelung transformation by setting $\gamma=\sqrt{-1}$ in \eqref{GHC} with general potential energies. Denote $\calF(\rho)=\sqrt{-1}\tilde{\calF}(\rho)$, where $\mathcal{F}\colon \mathcal{P}_+(M)\rightarrow \mathbb{R}$. Notice $\gamma^2=-1$, then the Hamiltonian system on density manifold \eqref{optimal_S1} forms
$$ \partial_t\rho_t=\delta_S\mathcal{H}(\rho_t, S_t),\quad \partial_tS_t=-\delta_\rho\mathcal{H}(\rho_t, S_t),$$    
with 
\begin{equation*}
\begin{split}
\mathcal{H}(\rho, S)=&\frac{1}{2}\int\abs{\nabla S(x)}^2\rho(x) dx-\frac{1}{2}\int\abs{\nabla\delta{\mathcal{F}}(\rho)(x)}^2\rho(x) dx\\
=&\frac{1}{2}\int\abs{\nabla S(x)}^2\rho(x) dx+\frac{1}{2}\int\abs{\nabla\delta\tilde{\mathcal{F}}(\rho)(x)}^2\rho(x) dx.\\
\end{split}
\end{equation*}
After the transformation \eqref{eq:def-GHC}, $(\eta_t, \eta^*_t)$ satisfies 
\begin{equation}\label{GSE}
\left\{\begin{aligned}
\partial_t\eta_t = &\sqrt{-1}\sigma(\eta_t,\eta^*_t) \partial_{\eta^*} \calK(\eta_t,\eta^*_t), \\
\partial_t\eta^*_t = &-\sqrt{-1}\sigma(\eta_t,\eta^*_t) \partial_{\eta} \calK(\eta_t,\eta^*_t).
\end{aligned}\right.
\end{equation}
Here $\eta^*_t$ is the complex conjugate of $\eta_t$, denoted by $\eta_t^*=\bar \eta_t$. Equation \eqref{GSE} can be viewed as the generalized Schr{\"o}dinger equations. It can also be written into one single equation:
\begin{equation*}
\partial_t\eta_t = \sqrt{-1}\sigma(\eta_t,\bar\eta_t) \partial_{\bar\eta} \calK(\eta_t,\bar\eta_t).
\end{equation*}
\begin{example}[Complex shallow-water equation]
In particular, consider the energy $\mathcal{F}(\rho)=\frac{\sqrt{-1}}{2}\int\rho^2 dx$.
Then after the generalized Madelung transformation, we derive the equation of $(\eta_t,\eta_t^*)$:
\begin{equation*}
-\sqrt{-1}\,\partial_t\eta_t = \frac{1}{2}\abs{\nabla\eta_t}^2 + (\eta_t+\bar\eta_t)\laplacian\eta_t .
\end{equation*}
We call it the complex shallow-water equation. 
\end{example}


\section{Energy splitting}
\label{section4}
In this section, we present an energy-splitting approach based on Hopf--Cole transformations. We present several inequalities for the split energies. 

We first define the class $\Ca$ of $a$-homogeneous functional.
\begin{definition}
\label{def:Ca}
Let $a>0$. We say that a smooth functional $\calF\colon \calP(M)\to\mathbb{R}$ is $a$-homogeneous, which we write $\calF\in\Ca$, if there exists $b\in\mathbb{R}$ such that
\begin{equation}
\label{eq:def-Ca}
\calF(\rho) = a^{-1}\int \delta\calF(\rho)\rho\,dx + b
\end{equation}
for all smooth probability densities $\rho$. 
\end{definition} 
We are now able to define a certain splitting of $\calF$ into two functionals in phase space.

\begin{definition}[Energy splitting]
Assume that $\calF\in \Ca$, and consider the Hopf--Cole transformation $s\colon (\eta,\eta^*)\to(\rho, S)$ defined in the previous section by~\eqref{eq:def-GHC}. We define $\calG$ and  $\calG^*$ on $\calC(M)$ by 
\begin{align*}
\calG(\eta,\eta^*) &= a^{-1}\int (\delta\mathcal{F})^{-1}(\delta\mathcal{F}(\eta)+\delta\mathcal{F}(\eta^*)) \,  \delta\mathcal{F}(\eta)\,dx + b/2,\\
\calG^*(\eta,\eta^*) &= a^{-1}\int (\delta\mathcal{F})^{-1}(\delta\mathcal{F}(\eta)+\delta\mathcal{F}(\eta^*)) \, \delta\mathcal{F}(\eta^*)\,dx + b/2 .
\end{align*}
The constant $b$ is the one appearing in Definition~\ref{def:Ca}, i.e. $b = \calF(\rho) - a^{-1}\int \delta\calF(\rho)\rho\,dx$.
\end{definition}

For any $\mathcal{F}\in\mathcal{C}(a)$ and from the definition of our Hopf--Cole formulation, it is clear that $\calF(\rho) = \calG(\eta,\eta^*) + \calG^*(\eta,\eta^*)$. This is the reason we call $\mathcal{G}$, $\mathcal{G}^*$ energy splitting functionals.

From now on, we shall demonstrate that along the solution of the SBP, $\mathcal{G}$, $\mathcal{G}^*$ share similar properties to the gradient descent and the gradient ascent flows of $\calF$. We shall prove energy dissipation results for each of them. 

\begin{theorem}[Energy splitting inequalities for Hopf--Cole transformation]\label{thm:ES}
If $\mathcal{F}$ is $\lambda$-convex in density manifold, then 
\begin{equation*}
\calG(\eta_t, \eta^*_t)+ct\calH \le \alpha_{1-t} \,\calG(\eta_0, \eta^*_0) + (1-\alpha_{1-t}) \,(\calG(\eta_1, \eta^*_1)+c\calH), 
\end{equation*}
and
\begin{equation*}
\calG^*(\eta_t, \eta^*_t)-ct\calH \le (1-\alpha_{t}) \,\calG(\eta_0, \eta^*_0) + \alpha_{t} \,(\calG(\eta_1, \eta^*_1)-c\calH).
\end{equation*}
Here 
\[
\alpha_{t} = \frac{1-e^{-2\lambda t}}{1-e^{-2\lambda}},
\]
$\calH$ denotes the Hamiltonian of the system (a constant of the flow) and $c=\frac{1-a^{-1}}{2}$ .
\end{theorem}

Let us mention that our study is motivated by the work~\cite{Conforti2018_second}, in which the author rigorously proved related energy splitting and convexity inequalities for the classical SBP. Theorem \ref{thm:ES} extends the  inequalities considered in \cite{Conforti2018_second} to our generalized SBP in a formal setting. Additionally, the work~\cite{GentilLeonardRipani2018_dynamical} proved in the context of the GSBP a convexity inequality on the potential $\calF$. Theorem~\ref{thm:ES} is very related to their result in the following sense: while~\cite{GentilLeonardRipani2018_dynamical} works directly on the potential $\calF$, we are more interested in our splitting $\calF(rho) = \calG(\eta,\eta^*) + \calG^*(\eta,\eta^*)$. 

The outline of the proof is as follows. We first compute the first derivative and second derivative along the Hamiltonian flow for the splitting energies. We then compare the value of first and second derivative of each splitting energy. Following Gr{\"o}nwall's inequality, we prove the entropy dissipation result.  

We first calculate the first derivative for the split energies $\mathcal{G}$, $\mathcal{G}^*$. 
\begin{lemma}[Energy production for Hopf--Cole transformation]\label{lemma:ES1}
The first time derivatives of split energies along the flow are
\begin{align*}
\frac{d}{dt}\calG(\eta_t,\eta^*_t) &= \int\rho_t\abs{\nabla \delta\mathcal{F}(\eta_t)}^2dx - c\,\calH ,\\
\frac{d}{dt}\calG^*(\eta_t,\eta^*_t) &= -\int\rho_t\abs{\nabla \delta\mathcal{F}(\eta^*_t)}^2dx +c\,\calH.
\end{align*}
Here $\calH$ denotes the Hamiltonian of the system (a constant of the flow) and $c=\frac{1-a^{-1}}{2}$ .
\end{lemma}

We next calculate the second derivative for the split energies $\mathcal{G}$ and $\mathcal{G}^*$.  
\begin{lemma}[Energy dissipation for Genearlized Hopf--Cole transformation]\label{lemma:ES2}

The second time-derivatives of the split energies along the flow are
\begin{equation*}
\begin{split}
\frac{d^2}{dt^2}\mathcal{G}(\eta_t, \eta^*_t)=&2\Hess_W\mathcal{F}(\rho_t)(V_{\delta\mathcal{F}(\eta_t)}, V_{\delta\mathcal{F}(\eta_t)}),\\
\frac{d^2}{dt^2}\mathcal{G}^*(\eta_t, \eta^*_t)=&2\Hess_W\mathcal{F}(\rho_t)(V_{\delta\mathcal{F}(\eta^*_t)}, V_{\delta\mathcal{F}(\eta^*_t)}),
\end{split}
\end{equation*}
where $\Hess_W$ is the Hessian operator in density manifold with 
\begin{equation*}
V_{\delta\mathcal{F}(\eta)}=-\div(\rho\nabla\delta\mathcal{F}(\eta)),
\end{equation*}
and 
\begin{equation*}
V_{\delta\mathcal{F}(\eta^*)}=-\div(\rho\nabla\delta\mathcal{F}(\eta^*)).
\end{equation*}
\end{lemma}

\begin{proof}[Proof of Lemma \ref{lemma:ES1}]
We begin the proof by noting that expressing $\delta\mathcal{F}(\eta)$ in terms of $\rho$ and $S$, and using the fact that $\calF\in \Ca$, the function $\calG$ defined in the theorem can be written in $(\rho, S)$ variables as 
\[
\calG = \frac{1}{2}\calF(\rho) + \frac{a^{-1}}{2}\int S\,\rho dx .
\]
Here we will abuse notation and write $\calG$ whether considering it as a function of $(\rho, S)$ or a function of $(\eta,\eta^*)$. In the proof we will also make use of the following results whose proofs are easy and left to the reader.

\begin{enumerate}[(i)]

\item If $\calF\in \Ca$ then the first variation $\delta \calF$ belongs to $\calC(a-1)$, in the sense that
\[
(a-1) \,\delta\calF(\rho)(x) = \int \rho(y)\,\delta^2\calF(\rho)(x,y)\,dy 
\]
for any $x\in M$. Here $\delta^2\calF$ denotes the second variation of $\calF$.
\item \label{lemma:class-ca-J} If $\calF\in \Ca$ then the square norm of the Wasserstein gradient of $\calF$
\[
\calJ(\rho) = \int \abs{\nabla\delta \mathcal{F}(\rho)}^2\rho
\]
belongs to $\calC(2a-1)$.
\end{enumerate}
Along the Hamiltonian flow~\eqref{optimal_S1} the time-derivative of the first term in the new expression of $\calG$ is simply
\[
\frac{d}{dt}\frac{1}{2}\calF(\rho_t) = \frac{1}{2}\int (\nabla \delta\mathcal{F}(\rho_t), \nabla S_t)\,\rho_t dx \,.
\]
For the second term, we compute after an integration by parts
\begin{align*}
\frac{d}{dt} \int S_t\,\rho_t dx &= \int \frac{1}{2}\abs{\nabla S_t}^2\,\rho_t dx + \int \delta J(\rho_t)\,\rho_t dx\\
 &= \int \frac{1}{2}\abs{\nabla S_t}^2\,\rho_t dx + (2a-1)J(\rho_t)
\end{align*}
where we use result~\ref{lemma:class-ca-J} above on $J$. Writing explicitly the expression of $J$ yields $$\frac{d}{dt} \int S_t\,\rho_t dx=\int \frac{1}{2}\abs{\nabla S_t}^2\,\rho_t dx+ (a-1/2)\int\abs{\nabla \delta\mathcal{F}(\rho_t)}^2\,\rho_t dx,$$ which implies that
\begin{equation*}
\begin{split}
\frac{d}{dt}\calG(\eta_t, \eta^*_t) =& \frac{a^{-1}}{2} \int\frac{1}{2}\abs{\nabla S_t}^2\,\rho_t dx + \left(1-\frac{a^{-1}}{2}\right) \int\frac{1}{2}\abs{\nabla \delta\mathcal{F}(\rho_t)}^2\,\rho_t dx\\
&+ \frac{1}{2}\int (\nabla \delta\mathcal{F}(\rho_t), \nabla S_t)\,\rho_t dx.
\end{split}
\end{equation*}
Introducing the Hamiltonian $\calH(\rho, S)= \int\frac{1}{2}\abs{\nabla S}^2\,\rho - \frac{1}{2}\abs{\nabla \delta\mathcal{F}(\rho)}^2\,\rho dx$, which is a conserved quantity. We can more simply write
\begin{align*}
\frac{d}{dt}\calG(\eta_t, \eta^*_t) &= \frac{1}{4}\int \abs{\nabla S_t + \nabla \delta\mathcal{F}(\rho_t)}^2\,\rho_t dx + \frac{a^{-1}-1}{2}\calH \\
 &= \int\abs{\nabla \delta\mathcal{F}(\eta_t)}^2\,\rho_t dx + \frac{a^{-1}-1}{2}\calH
\end{align*}
where we used that $\delta\mathcal{F}(\eta) = \frac{S + f(\rho)}{2}$. By a similar computation one can obtain the expression of $\frac{d}{dt}\calG^*$.
\end{proof}

\begin{proof}[Proof of Lemma \ref{lemma:ES2}]
The proof is based on the Riemannian calculus in Wasserstein density manifold. For readers who are not familiar with infinite dimensional geometry calculus, one can find the finite dimensional analog provided in Lemma \ref{lemma:manifolds:second-derivative}. We represent $(\eta_t, \eta_t^*)$ in coordinates $(\rho_t, S_t)$, so as $(\rho_t, \partial_t\rho_t)$. From the transformation \eqref{eq:def-GHC}, notice the fact that $\delta\mathcal{F}(\eta_t) = \frac{1}{2}\big(S_t + \delta\mathcal{F}(\rho_t)\big)$ and $\partial_t\rho_t=-\Delta_{\rho_t}S_t$, thus denote  $$a_t=\partial_t\rho_t+\textrm{grad}_W\mathcal{F}(\rho_t).$$
From \eqref{SOE}, we can simply check 
\begin{equation*}
\begin{split}
\frac{D}{dt}a_t=&\frac{D^2}{dt^2}\rho_t+\frac{D}{dt}\textrm{grad}_W\mathcal{F}(\rho_t)\\
=&\frac{1}{2}\textrm{grad}_W(\textrm{g}_\rho^W(\textrm{grad}_W\mathcal{F}(\rho), \textrm{grad}_W\mathcal{F}(\rho_t)))+\textrm{Hess}_W\mathcal{F}(\rho_t)\partial_t\rho_t\\
=&\textrm{Hess}_W\mathcal{F}(\rho_t)(\textrm{grad}_W\mathcal{F}(\rho_t)+ \partial_t\rho_t)\\
=&\textrm{Hess}_W\mathcal{F}(\rho_t)a_t.
\end{split}
\end{equation*}
where $\frac{D}{dt}$ is the covariant derivative in density manifold. We now recall the result from Lemma~\ref{lemma:manifolds:first-derivative}:
\begin{align*}
\frac{d}{dt}\mathcal{G}(\eta_t, \eta_t^*) =& \frac{1}{4}\int\abs{\nabla \delta\mathcal{F}(\eta_t)}^2\,\rho_t dx + \frac{a^{-1}-1}{2}\calH\\
=& \frac{1}{4}\int  \big(\delta\mathcal{F}(\eta_t), (-\Delta_{\rho_t})\delta\mathcal{F}(\eta_t)\big)\rho_t dx + \frac{a^{-1}-1}{2}\calH\\
=& \frac{1}{4}\int  \big(\Delta_{\rho_t}\delta\mathcal{F}(\eta_t), (-\Delta_{\rho_t})^{-1} \Delta_{\rho_t}\delta\mathcal{F}(\eta_t) \big)dx+ + \frac{a^{-1}-1}{2}\calH\\
=&g^W_{\rho_t}(a_t, a_t)+ \frac{a^{-1}-1}{2}\calH.
\end{align*}
We are ready to compute the second time-derivative of $\mathcal{G}$:
\begin{equation*}
\frac{d^2}{dt^2}\mathcal{G}(\eta_t, \eta^*_t)= \frac{d}{dt}g^W_{\rho_t}(a_t, a_t)
=2\int g_{\rho_t}^W(\frac{D}{dt}a_t, a_t)=2\textrm{Hess}_W\mathcal{F}(\rho_t)(a_t, a_t).
\end{equation*}
From the Hessian formula derived in Wasserstein geometry \cite{Li2018_geometrya}, we prove the result. 
 A similar computation can be used to obtain the expression of $\frac{d}{dt}\mathcal{G}^*$. 
\end{proof}

By combining Lemma~\ref{lemma:ES1} and~\ref{lemma:ES2}, we now proceed with proving Theorem~\ref{thm:ES}. 
\begin{proof}[Proof of Theorem \ref{thm:ES}]
From Lemma~\ref{lemma:ES1} and Lemma~\ref{lemma:ES2}, along the Hamiltonian flow \eqref{optimal_S1}, we have that

\[
\frac{d^2}{dt^2}\mathcal{G}(\eta_t,\eta^*_t) \ge 2\lambda \left(\frac{d}{dt}\mathcal{G}(\eta_t, \eta^*_t) + c\calH\right),
\]
and
\[
\frac{d^2}{dt^2}\mathcal{G}^*(\eta_t, \eta^*_t) \ge -2\lambda \left(\frac{d}{dt}\mathcal{G}^*(\eta_t, \eta^*_t) - c\calH\right).
\]
Integrating in time variable $[0,t]$, and applying the Grownwall's inequality, we finish the proof. 
\end{proof}

\begin{remark}
Let $\calA(\mu,\nu)$ be the value of the generalized SBP problem~\eqref{eq:GSB2}. Theorem~\ref{thm:ES} allows us to express $\calA(\mu,\nu)$ in terms of variables at initial and final times only. Indeed, note that along the Hamiltonian flow~\eqref{optimal_S1}, 
\begin{align*}
\frac{d}{dt}\left(\calG - \calG^*\right) &= \int \big(\abs{\nabla \delta\mathcal{F}(\eta_t)}^2 + \abs{\nabla \delta\mathcal{F}(\eta^*_t)}^2\big)\,\rho_t dx - 2c\,\calH \\
 &= \int \frac{1}{2} \big(\abs{\nabla S_t}^2 + \abs{\nabla \delta\mathcal{F}(\rho_t)}^2\big)\rho_t dx- 2c\,\calH \\
 &= \calL(\rho_t, S_t) - 2c\,\calH,
\end{align*}
where $\calL$ denotes the \emph{Lagrangian} of the problem. Since, \emph{along the optimal flow}~\eqref{optimal_S1} with the boundary conditions $\rho_0=\mu, \rho_1=\nu$, we have that $\calA(\mu,\nu)=\int_0^1 \calL(\rho_t,S_t)\,dt$, we deduce that
\[
\calA(\mu,\nu) = \calG(\nu,S_1) - \calG(\mu,S_0) + \calG^*(\mu,S_0) - \calG^*(\nu,S_1) +2c\,\calH,
\]
where we recall that $\calH=\int \frac{1}{2} \abs{\nabla S}^2\,\rho - \frac{1}{2}\abs{\nabla \delta\mathcal{F}(\rho)}^2\,\rho dx$ can be taken at any time $t\in [0,1]$ since it is a constant of motion. We can further simplify the previous expression using the identity $\calF = \calG + \calG^*$, which implies that
\[
\calA(\mu,\nu) = 2\,\calG(\nu,S_1) + 2\,\calG^*(\mu,S_0) - \calF(\mu)-\calF(\nu) +2c\,\calH .
\]
Naturally, the value of $S_0$ or $S_1$ is not known before fully solving the problem~\eqref{eq:GSB2}. See similar discussions in \cite{Conforti2018_second} and \cite{GentilLeonardRipani2018_dynamical}.
\end{remark}

\subsection{Examples}
Many functionals usually considered in optimal transport and information theory \cite{IG2} are homogeneous: this fact is highlighted below, where we examine several important examples.

\begin{example}[Rényi entropy]
It is easy to show that the Rényi entropy 
\[
\calF(\rho) = \frac{\gamma}{m+1} \int \rho^{m+1}\,dx
\]
is $(m+1)$-homogeneous, in the sense of Definition~\ref{def:Ca}, and we can therefore apply the splitting result of Theorem~\ref{thm:ES}. The expressions of $\calG$ and $\calG^*$ are  
\[
\calG(\eta,\eta^*) = \frac{\gamma}{m+1}\int \eta^m\,\rho\, dx
\]
and 
\[
\calG^*(\eta,\eta^*) = \frac{\gamma}{m+1}\int \eta^{*m}\,\rho \,dx,
\]
where like before $\rho$ is understood as a function of $(\eta,\eta^*)$, i.e. here $\rho = (\eta^m + \eta^{*m})^{1/m}$. It is easy to check that indeed $\calF(\rho) = \calG(\eta,\eta^*) + \calG^*(\eta,\eta^*)$. Moreover, along the flow~\eqref{eq:HF-Renyi-entropy}, the first time-derivatives of $\calG$ and $\calG^*$ are 
\begin{gather*}
\frac{d}{dt}\calG(\eta_t,\eta^*_t) = \gamma^2 \int\abs{\nabla \eta_t^m}^2\,\rho_t dx - \,c\calH , \\
\frac{d}{dt}\calG^*(\eta_t,\eta^*_t) = -\gamma^2 \int\abs{\nabla \eta_t^{*m}}^2\,\rho_t dx + c\,\calH ,
\end{gather*}
where $c=\frac{m}{2(m+1)}$ and where the Hamiltonian expressed in $(\eta,\eta^*)$ variables is $\calH = -2\gamma^2\int \nabla (\eta^m)\cdot\nabla(\eta^{*m})\,\rho dx$. We recall that the Hamiltonian is a conserved quantity and therefore $\calH$ is constant in time. 
\end{example}

\begin{example}[Interaction energy]
It is easy to check that interaction energies are $2$-homogeneous. Therefore, applying Theorem~\ref{thm:ES} we can split $\calF = \calG + \calG^*$ where
\begin{gather*}
\calG(\eta, \eta^*) = \frac{1}{2}\int (\eta + \eta^*)\,W\conv \eta\, dx, \\
\calG^*(\eta, \eta^*) = \frac{1}{2}\int (\eta + \eta^*)\,W\conv \eta^* \,dx .
\end{gather*}
Along the Hamiltonian flow~\eqref{optimal_S1}, the variation of $\calG$ and $\calG^*$ is
\begin{gather*}
\frac{d}{dt}\calG(\eta_t,\eta^*_t) = \int (\eta_t + \eta^*_t)\,\abs{\nabla W\conv\eta_t}^2 dx - \calH / 4 , \\
\frac{d}{dt}\calG^*(\eta_t,\eta^*_t) = -\int (\eta_t + \eta^*_t)\,\abs{\nabla W\conv\eta^*_t}^2 dx + \calH / 4,
\end{gather*}
where the Hamiltonian in $(\eta,\eta^*)$ variables can be computed to be $$\calH(\eta, \eta^*) = -2\int (\eta + \eta^*)\, \nabla (W\conv\eta) \cdot \nabla (W\conv\eta^*) dx.$$ 
\end{example}
In fact, there are lots of interesting Hessian formulas in Wasserstein space. For example, the Hessian operator of linear entropy connects with the Bakery--Emery Gamma two operator \cite{Villani2003, Li2018_geometrya}. From the associated smallest eigenvalue of Hessian operators, one can derive related inequalities for the split energies.

\section{Finite-dimensional analogues}\label{section5}
In this section our aim is to show that most results presented in this paper are not only true for the density manifold with Wasserstein metric, but rather are verified on any (finite-dimensional) manifold whose metric satisfies some properties detailed below. We therefore prove analogues of most result presented in the paper and follow the same outline.

Our setting in this section is a finite-dimensional Riemannian manifold $(\calM,g)$ together with global coordinates $(q^i)$ on $\calM$, as well as a smooth potential function
\[
F\colon \calM\to\mathbb{R} .
\]
For ease of notation, we will often write the metric as $\abs{v}^2$ instead of $g(v,v)$, if $v$ is a vector or a co-vector. Moreover we will denote gradients with the symbol $\nabla$. Morally $\calM$ corresponds to the Wasserstein space and the global coordinates are the $L^2$ coordinates. We refer to~\cite{Leger2018} for more details on this setting. From now on, we use Einstein summation symbol freely. 

\subsection{Finite-dimensional Schr{\"o}dinger bridge problems}

The manifold analogue of the GSBP~\eqref{eq:GSB2} is the controlled gradient flow problem
\begin{equation}
\label{eq:gsb:functional}
\inf_{q,b} \int_0^1\frac{1}{2}\abs{b_t}^2\,dt
\end{equation}
where the infimum runs over smooth paths $(q,b)\colon[0,1]\to T\calM$ constrained by
\begin{equation}
    \label{eq:gsb:constraint}
\dot{q}_t = b_t - \nabla F(q_t)
\end{equation}
for all $t\in (0,1)$ 
with the boundary conditions
\[
q_0=x, \quad q_1=y,
\]
where $x$ and $y$ are fixed point in $\calM$. This problem was introduced in~\cite{Leger2018} and also studied in~\cite{GentilLeonardRipani2018_dynamical}. There exists an equivalent, ``Lagrangian-mechanics'' version of this problem given by
\begin{equation}
\label{eq:gsb-lagrangian}
\inf_{q,v} \int_0^1\frac{1}{2}\abs{v_t}^2+\frac{1}{2}\abs{\nabla F(q_t)}^2\,dt +F(y)-F(x),
\end{equation}
where $v$ denotes the velocity, $\dot{q}=v$. The same boundary conditions as above are considered. We refer to~\cite{Leger2018} for more detailed explanations. Associated to this problem are the Lagrangian $L(q,v)=\frac{1}{2}\abs{v}^2+\frac{1}{2}\abs{\nabla F(q)}^2$, as well as the Hamiltonian
\begin{equation}
\label{eq:manifolds:H}
H(q,p)=\frac{1}{2}\abs{p}^2-\frac{1}{2}\abs{\nabla F(q)}^2 .
\end{equation}
Here the momentum $p\in T^*_q\calM$ is the covector associated to the velocity $v$: $p_i=g_{ij}v^j$.

The optimally conditions of the problem~\eqref{eq:gsb-lagrangian} are, in coordinates:
\begin{equation}
\label{eq:manifolds:HF-q}
\begin{cases}
\dot{q}^i = g^{ij}p_j, \\
\dot{p}_i = -\frac{1}{2}\partial_ig^{jk}p_jp_k + \frac{1}{2} \partial_i(g^{jk}\partial_jF\partial_kF),
\end{cases}
\end{equation}
completed with the boundary conditions.


\subsection{Symplectic aspects}
Our Hopf--Cole transformation $s\colon (\eta,\eta^*)\to (q,p)$, introduced in~\cite{Leger2018}, was also defined on manifolds by 
\begin{equation}
\label{eq:def-ghc}
\begin{cases}
\partial_iF(\eta) = \frac{1}{2}\big(p_i+\partial_iF(q)\big) \\
\partial_iF(\eta^*) = \frac{1}{2} \big(-p_i+\partial_iF(q)\big) ,
\end{cases}
\end{equation}
provided $\partial_iF$ is invertible. 

Before stating an analogue of Lemma~\ref{lemma}, we first introduce some notation. We write $f_i(q)=\partial_iF(q)$ for the first derivative of $F
$ in coordinates and $h_{ij}(q)=\partial^2_{ij}F(q)$ for the second derivative. We also write $h^{ij}$ for the inverse tensor of $h_{ij}$.

\begin{lemma}
\label{lemma:manifolds:HF-eta}
Consider the Hopf--Cole transformation $s\colon (\eta,\eta^*)\to (q,p)$ defined by~\eqref{eq:def-ghc}. Let
\[
\sigma^{i\ell}(\eta, \eta^*) = \frac{1}{2}h^{ij}(\eta) h_{jk}(q) h^{k\ell}(\eta^*) 
\]
where $q$ stands for $f^{-1}\big(f(\eta)+f(\eta^*)\big)$. Then, the Hamiltonian flow~\eqref{eq:manifolds:HF-q} can be written in the new variables as
\begin{equation*}
\begin{cases}
\dot{\eta}^i = -\sigma^{ij}(\eta, \eta^*) \frac{\partial K}{\partial \eta^{*j}} \\
\dot{\eta}^{*i} = \sigma^{ji}(\eta, \eta^*) \frac{\partial K}{\partial \eta^j}
\end{cases}
\end{equation*}
Here $K$ denotes the Hamiltonian in the new variables:
\[
K(\eta,\eta^*) = H\big(s(\eta,\eta^*)\big) = -2\,g^{ij}(q)\,f_i(\eta)\,f_j(\eta^*)
\]
where we write $q=f^{-1}(f(\eta) + f(\eta^*))$.
\end{lemma}

\begin{remark}
The expression of $\sigma$ doesn't depend on the metric $g$. 
\end{remark}

\begin{theorem}
Assume that in coordinates $(q^i)$ the potential $F$ is written 
\[
F(q) = \frac{1}{2} W_{ij}q^i q^j + U_i \,q^i ,
\]
where $W$ is a symmetric positive-definite matrix and $U$ is a vector in $\Rn$. Then the Hamiltonian flow~\eqref{eq:manifolds:HF-q} can be written in the new variables as
\begin{equation*}
\begin{cases}
\dot{\eta}^i = -2 W^{ij} \frac{\partial K}{\partial \eta^{*j}} \\
\dot{\eta}^{*i} = 2W^{ij} \frac{\partial K}{\partial \eta^j}
\end{cases}
\end{equation*}
where $W^{ij}$ denotes the inverse of the matrix $W_{ij}$. 
\end{theorem}

\begin{example}[Flat metric]
Consider the simple case where $\calM=\Rn$ is equipped with the flat metric $g_{ij}(q) = \delta_{ij}$. Let 
\[
F(q)=\frac{1}{2} \bracket{A \,q, q},
\]
where $A$ is a symmetric positive-definite $n\times n$ matrix and where $\bracket{\cdot,\cdot}$ denotes the canonical inner product on $\Rn$. The Hamiltonian associated to $F$ is 
\[
H(q,p) = \frac{1}{2}\norm{p}^2 - \frac{1}{2}\norm{Aq}^2
\]
where $\norm{\cdot}$ is the Euclidean norm, and the Hamiltonian flows equations are 
\begin{equation*}
\begin{cases}
\dot{q} = p, \\
\dot{p} = A^2q .
\end{cases}
\end{equation*}
It is easy to check that our Hopf--Cole transformation is given by $A q = A \eta + A \eta^*$, $p = A\eta - A\eta^*$, i.e.
\begin{equation*}
\left\{
\begin{aligned}
\eta = \frac{1}{2}(q + A^{-1}p), \\
\eta = \frac{1}{2}(q - A^{-1}p) .
\end{aligned}\right.
\end{equation*}
The flow equations in the new variables are 
\begin{equation*}
\begin{cases}
\dot{\eta} = A\eta, \\
\dot{\eta^*} = -A\eta^* .
\end{cases}
\end{equation*}

Note that in this simple case the $\eta$ and $\eta^*$ equations are respective gradient ascent and descent flows of $F$. This phenomena is explained in details in~\cite{Leger2018}.
\end{example}

\begin{proof}[Proof of Lemma~\ref{lemma:manifolds:HF-eta}]
We recall the notation used in the lemma: $f_i(q) = \partial_iF(q)$ and $h_{ij}(q) = \partial^2_{ij}F(q)$. Let $t\to(q_t,p_t)\in T^*\calM$ be a solution to the Hamiltonian flow~\eqref{eq:manifolds:HF-q}, which we recall takes the form
\begin{equation*}
\begin{cases}
\dot{q}^i = g^{ij}(q)p_j, \\
\dot{p}_i = -\frac{1}{2}\partial_ig^{jk}(q)p_jp_k + \frac{1}{2} \partial_i(g^{jk}\partial_jF\partial_kF)(q).
\end{cases}
\end{equation*}
Here we don't need to fix boundary conditions. We would like to derive an equation on $\dot{\eta}$, where $\eta$ is defined by our generalized Hopf--Cole transformation: in other words $f_j(\eta) = \frac{1}{2}\big(p_j + f_j(q)\big)$. Taking time-derivatives on both sides of this equality implies
\[
h_{ij}(\eta)\,\dot{\eta}^i = \frac{1}{2}\dot{p}_j + \frac{1}{2}\frac{d}{dt}f_j(q) .
\]
We now compute each term in the RHS separately. Firstly,
\begin{align*}
\dot{p}_j &= -\frac{1}{2}\partial_jg^{k\ell}(q)p_kp_{\ell} + \frac{1}{2} \partial_j(g^{kl}f_k f_{\ell})(q) \\
 &= -\frac{1}{2}\partial_jg^{k\ell}(q)p_kp_{\ell} + \frac{1}{2}\partial_jg^{k\ell}(q)f_k(q) f_{\ell}(q) + g^{k\ell}(q) h_{jk}(q)f_{\ell}(q) .
\end{align*}
Note that by symmetry of indices $k$ and $\ell$, the first two terms can be factorized into the expression $\frac{1}{2}\partial_jg^{k\ell}(q)\big(p_k + f_k(q)\big)\big(-p_{\ell} + f_{\ell}(q)\big)$. Using the expression~\eqref{eq:def-ghc} which defines $\eta$ and $\eta^*$, we obtain
\[
\dot{p}_j = 2\,\partial_jg^{k\ell}(q) f_k(\eta) f_{\ell}(\eta^*) + g^{k\ell}(q) h_{jk}(q)f_{\ell}(q) .
\]
Secondly, we can easily check that $\frac{d}{dt}f_j(q) = h_{jk}(q)\dot{q}^k = h_{jk}(q) g^{k\ell}(q)p_{\ell}$. Therefore, combining the expression obtained thus far, one can check that $\dot{p}_j + \frac{d}{dt}f_j(q) = 2\,\partial_jg^{k\ell}(q) f_k(\eta) f_{\ell}(\eta^*) + 2\,g^{k\ell}(q) h_{jk}(q)f_{\ell}(\eta)$. This implies an expression for $\dot{\eta}$,
\[
h_{ij}(\eta)\,\dot{\eta}^i = \partial_jg^{k\ell}(q) f_k(\eta) f_{\ell}(\eta^*) + h_{jk}(q) g^{k\ell}(q) f_{\ell}(\eta) .
\]

Next, we would like to relate the expression of $\dot{\eta}$ to the Hamiltonian. Recall that the Hamiltonian is given by $H(q,p) = \frac{1}{2}\abs{p}^2 - \frac{1}{2}\abs{\nabla F(q)}^2$, which in coordinates takes the form
\[
H(q,p) = \frac{1}{2}g^{jk}(q)p_jp_k - \frac{1}{2}g^{jk}(q) f_j(q) f_k(q) .
\]
Similarly to an operation above, by symmetry of indices $j$ and $k$ the above difference can be factorized into $\frac{1}{2}g^{jk}(q)\big(p_j + f_j(q)\big)\big(p_k - f_k(q)\big)$. Switching to $(\eta,\eta^*)$ variables, we can obtain the expression of the Hamiltonian in the new variables, 
\[
K(\eta,\eta^*) = -2 \,g^{jk}(q) f_j(\eta) f_k(\eta^*).
\]
Here $q$ should be implicitly understood as a function of $(\eta,\eta^*)$ (i.e.  defined by $f_i(q) = f_i(\eta) + f_i(\eta^*)$ ). Let us now compute a partial derivative:
\begin{align*}
\frac{\partial K}{\partial \eta^{*\ell}} &= -2 \left[ \partial_a g^{jk}(q)\frac{\partial q^a}{\partial \eta^{*\ell}} f_j(\eta)f_k(\eta^*) + g^{jk}(q)f_j(\eta)h_{k\ell}(\eta^*) \right].
\end{align*}
By differentiating the expression $f_i(q) = f_i(\eta) + f_i(\eta^*)$ with respect to $\eta^*$, it is clear that 
\[
\frac{\partial q^a}{\partial \eta^{*\ell}} = h^{ak}(q) h_{k\ell}(\eta^*) .
\]
Note that here the tensor $h^{ij}$ denotes the inverse of $h_{ij}$. As a consequence we derive the identity
\[
-\frac{1}{2} h^{k\ell}(\eta^*) \frac{\partial K}{\partial \eta^{*\ell}}= h^{jk}(q) \,\partial_jg^{k\ell}(q) f_k(\eta)f_{\ell}(\eta^*) + g^{k\ell}(q)f_{\ell}(\eta) ,
\]
where we have changed the name of some repeated indices. By multiplying both sides by $h_{jk}(q)$, we obtain
\[
h_{ij}(\eta)\,\dot{\eta}^i = -\frac{1}{2} h_{jk}(q) h^{k\ell}(\eta^*) \frac{\partial K}{\partial \eta^{*\ell}}, 
\]
which can be rearranged into the desired expression, $$\dot{\eta}^i = -\frac{1}{2} h^{ij}(\eta)h_{jk}(q) h^{k\ell}(\eta^*) \frac{\partial K}{\partial \eta^{*\ell}}.$$ 

An almost identical line of proof can be used to obtain the required expression for $\dot{\eta}^*$. 

\end{proof}

We now derive the symplectic matrix for the generalized Hopf--Cole transform.
\begin{lemma}
Define on $\Rn\times\Rn$ the symplectic form $\omega$ by
\[
\omega(\eta,\eta^*) = \begin{pmatrix}
0 & -\sigma(\eta,\eta^*) \\
\sigma^\mathsf{T}(\eta,\eta^*) & 0
\end{pmatrix},
\]
where $\sigma$ is the coefficient defined in Lemma~\ref{lemma:manifolds:HF-eta}. Then our Hopf--Cole transformation is a symplectomorphism between $(\Rn\times\Rn, \om)$ and $T^*M$ equipped with its natural symplectic form. 
\end{lemma}

With this symplectic perspective, the result in Lemma~\ref{lemma:manifolds:HF-eta} can be written more compactly as
\[
\frac{d}{dt} (\eta_t,\eta^*_t) = \nabla_{\omega}K(\eta_t,\eta^*_t),
\]
where $\nabla_{\om}K$ is the \emph{symplectic} gradient of $K$. It is the vector field defined by $\om(\nabla_{\om}K, u) = DK(u)$ for any vector field $u$, where $DK$ stands for the differential map of $K$.

\subsection{Energy splitting}
This section follows very closely the one in Wasserstein space, roughly giving equivalent results for each one presented in Section~\ref{section4}. 

Note first that the class $\Ca$ of $a$-homogeneous potentials, introduced in Section~\ref{section4}, is a notion that still makes sense on manifolds provided we work in coordinates:

\begin{definition}[Homogeneous functions on manifolds]
\label{def:manifolds:Ca}
Let $a>0$. We say that a smooth function $F\colon M\to\Real$ is $a$-homogeneous, and we write $F\in \Ca$, if there exists $b\in\Real$ such that 
\[
F(q) = a^{-1} q^i\partial_i F(q) + b
\]
for all $q\in M$. 
\end{definition}

Note that this definition only depends on the expression of $F$ in coordinates but not on the metric $g$ itself. We now define our energy splitting $F = G + G^*$. 

\begin{definition}
\label{def:manifolds:G}
Assume that $F\in\Ca$. We define the function $G$ and $G^*$ by
\[
G(\eta,\eta^*) = a^{-1} q^i f_i(\eta) + b/2
\]
and 
\[
G^*(\eta,\eta^*) = a^{-1} q^i f_i(\eta^*) + b/2 .
\]
Here $q^i$ is to be understood as a function of $(\eta,\eta^*)$, i.e. inverting the equality $f_i(q) = f_i(\eta) + f_i(\eta^*)$. We recall that $f$ denotes the first derivative of the potential $F$: $f_i(q) = \partial_iF(q)$. Moreover, $b$ is the constant appearing in the definition of $a$-homogeneity: $b=F(q) - a^{-1}q^i f_i(q)$. Since $F$ is $a$-homogeneous, it is easy to see that
\[
F(q) = G(\eta,\eta^*) + G^*(\eta,\eta^*) .
\]
\end{definition}

The main theorem of this section is a direct analogue of Theorem~\ref{thm:ES} in Wasserstein space. It is available on manifolds when the metric $g$ satisfies, like the potential $F$, a homogeneity condition.

\begin{assumption}
\label{assumption:manifolds:F}
The potential $F$ is $a$-homogeneous, i.e. $F\in\Ca$. 
\end{assumption}

\begin{assumption}
\label{assumption:manifolds:g}
The metric $g$ satisfies, in coordinates $(q^i)$, the homogeneity condition 
\begin{equation*}
g^{jk}(q) = m\, q^i\partial_ig^{jk}(q)
\end{equation*}
for some $m > 0$. 
\end{assumption}
Note that in these two assumptions are implicitly chosen the global coordinate chart $(q^i)$. An important example is to consider $\mathcal{M}$ as the probability simplex $$\mathcal{M}=\{(q_i)_{i=1}^n\in \mathbb{R}^n\colon \sum_{i=1}^n q_i=1,~q_i\geq 0\}.$$ 
One can define Riemannian structures on probability simplex such as a Wasserstein metric, which satisfies the proposed assumptions. See details in subsection \ref{HCgraph}.  
\begin{theorem}
\label{th:manifolds:inequalities}
If $F$ is $\lambda$-convex with $\lambda\in\Real$, then along the flow

\begin{equation*}
G(\eta_t, \eta^*_t)+ctH \le \alpha_{1-t} \,G(\eta_0, \eta^*_0) + (1-\alpha_{1-t}) \,(G(\eta_1, \eta^*_1) + cH), 
\end{equation*}
and
\begin{equation*}
G^*(\eta_t, \eta^*_t)-ctH \le (1-\alpha_{t}) \,G(\eta_0, \eta^*_0) + \alpha_{t} \,(G(\eta_1, \eta^*_1) - cH) .
\end{equation*}
Here 
\[
\alpha_{t} = \frac{1-e^{-2\lambda t}}{1-e^{-2\lambda}},
\]
$H$ denotes the Hamiltonian of the system (a constant of the flow) and $c=\frac{a^{-1}(m-2)+1}{2}$ .
\end{theorem}

\begin{lemma}
\label{lemma:manifolds:first-derivative}
Suppose that Assumptions~\ref{assumption:manifolds:F} and~\ref{assumption:manifolds:g} hold. Consider the Hopf--Cole transformation $s\colon (\eta,\eta^*)\to (q,p)$ defined by~\eqref{eq:def-ghc}. The splitting $F(q) = G(\eta,\eta^*) + G^*(\eta,\eta^*)$ satisfies along the flow~\eqref{eq:manifolds:HF-q}

\begin{gather*}
\frac{d}{dt}G(\eta_t,\eta^*_t) = g^{ij}(q_t)\,f_i(\eta_t)\,f_j(\eta_t) - c\,H, \\
\frac{d}{dt}G^*(\eta_t,\eta^*_t) = -g^{ij}(q_t)\,f_i(\eta^*_t)\,f_j(\eta^*_t) + c\,H .
\end{gather*}
Here $H$ denotes the Hamiltonian of the system (a constant of the flow) and $c=\frac{a^{-1}(m-2)+1}{2}$ .
\end{lemma}

To prove the main theorem of this section, we not only need information on the first time-derivative of $G$ and $G^*$ but also on the second time-derivative. 

\begin{lemma}
\label{lemma:manifolds:second-derivative}
Suppose that Assumptions~\ref{assumption:manifolds:F} and~\ref{assumption:manifolds:g} hold. Then, the two functions $G$ and $G^*$ introduced in Definition~\ref{def:manifolds:G} satisfy along the flow~\eqref{eq:manifolds:HF-q}
\begin{gather*}
\frac{d^2}{dt^2}G = 2\Big(\partial^2_{k\ell}F(q_t) - \Gamma^n_{k\ell}(q_t) \,f_n(q_t)\Big) g^{ik}(q_t)g^{j\ell}(q_t)f_i(\eta_t)f_j(\eta_t), \\
\frac{d^2}{dt^2}G^* = 2\Big(\partial^2_{k\ell}F(q_t) - \Gamma^n_{k\ell}(q_t) \,f_n(q_t)\Big) g^{ik}(q_t)g^{j\ell}(q_t)f_i(\eta^*_t)f_j(\eta^*_t) ,
\end{gather*}
where the $\Gamma^n_{k\ell}$'s denote the Christoffel symbols associated with the metric $g$. As a consequence, a lower bound on the Hessian of $F$ of the form
\begin{equation*}
\nabla^2F \ge \lambda\,g
\end{equation*}
where $\lambda\in\Real$ implies that
\begin{equation*}
\frac{d^2G}{dt^2} \ge 2\lambda\, g^{ij}(q_t)\,f_i(\eta_t)\,f_j(\eta_t)
\end{equation*}
and 
\begin{equation*}
\frac{d^2G^*}{dt^2} \ge 2\lambda\, g^{ij}(q_t)\,f_i(\eta^*_t)\,f_j(\eta^*_t) .
\end{equation*}
\end{lemma}

\begin{remark}
In Lemma~\ref{lemma:manifolds:second-derivative}, the term with Christoffel symbols is exactly the expression of the Hessian of $F$ in coordinates. More precisely, let $\nabla^2F$ be the Hessian of $F$, in the classical Riemannian sense. If $(e_i)$ denotes the basis on $TM$ associated with coordinates $(q^i)$, then 
\[
\nabla^2 F(q)(e_k, e_{\ell}) = \partial^2_{k\ell}F(q) - \Gamma^n_{k\ell} \partial_n F(q) .
\]
\end{remark}

\begin{proof}[Proof of Lemma~\ref{lemma:manifolds:first-derivative}]
Let $F$ be a potential in the class $\Ca$ and $g$ a metric satisfying Assumption~\ref{assumption:manifolds:g} for some $m > 0$. We start the proof by splitting $F$ into the sum $G+G^*$ according to Definition~\ref{def:manifolds:G}. Similarly to the Wasserstein proof, we note that $G$ can be written in $(q,p)$ variables as
\[
G = \frac{1}{2}F(q) + \frac{1}{2}a^{-1}q^i p_i .
\]
Here we abuse notation and write $G$ whether considering the function in $(q,p)$ variables or in the transformed $(\eta,\eta^*)$ variables. Therefore we compute

\begin{align*}
\frac{dG}{dt} &= \frac{1}{2}f_i(q)\dot{q}^i + \frac{1}{2}a^{-1}(\dot{q}^ip_i + q^i \dot{p}_i ) .
\end{align*}
For clarity, we will drop the $q$ dependence in expressions such as $f_i(q)$, $g^{ij}(q)$, etc. Using the expression of the time-derivatives given by~\eqref{eq:manifolds:HF-q} implies
\[
\frac{dG}{dt} = \frac{1}{2}g^{ij} f_i p_j + \frac{1}{2} a^{-1}\left[ g^{ij} p_i p_j - \frac{1}{2}q^i\partial_ig^{jk} p_jp_k + q^i\partial_i J\right] ,
\]
where $J(q) = \frac{1}{2} g^{jk}f_jf_k$. Because of the homogeneity condition on $g$ we can simplify the third term, $-\frac{1}{2}q^i\partial_ig^{jk} p_jp_k = -\frac{1}{2}mg^{jk} p_jp_k$. Next, we deal with the last term. We directly compute
\begin{align*}
q^i\partial_i J &= \frac{1}{2} q^i\partial_i (g^{jk}f_jf_k) \\
 &= \frac{1}{2}q^i\partial_i g^{jk} f_jf_k + q^i g^{jk}\partial^2_{ij}F f_k .
\end{align*}
Using once more the homogeneity condition on $g$ the first term can be written as $\frac{1}{2} m g^{jk} f_j f_k$. As for the second term, we can simplify it by noting that $F\in\Ca \implies \partial_iF\in\calC(a-1)$; more precisely
\[
q^i\partial_{ij}^2F = (a-1) \partial_j F .
\]
Therefore $q^i\partial_i J = \frac{1}{2} m g^{ij}f_if_j + (a-1)g^{ij} f_i f_j$. Grouping similar terms together, we have
\begin{align*}
\frac{dG}{dt} &= \frac{1}{2}g^{ij} f_i p_j -\frac{(m-2)a^{-1}}{4} g^{ij} p_i p_j + \frac{(m-2)a^{-1}+2}{4} g^{ij} f_i f_j \\
    &= g^{ij} \frac{f_i + p_i}{2}\frac{f_j + p_j}{2}-\frac{(m-2)a^{-1} + 1}{4} g^{ij} p_i p_j + \frac{(m-2)a^{-1}+1}{4} g^{ij} f_i f_j .
\end{align*}
To conclude the proof, note that the first term is exactly $g^{ij}(q) f_i(\eta)f_j(\eta)$ and the last two terms combine into $-c\, H(q,p)$ with $c = \frac{(m-2)a^{-1}+1}{2}$. 

A very similar line of proof can be used for the time-derivative of $G^*$. 
\end{proof}

\begin{proof}[Proof of Lemma~\ref{lemma:manifolds:second-derivative}]
For this proof it is best to stay away from working in coordinates. Since the content of the lemma is geometric in nature, it allows to work directly with geometric quantities (rather than having to describe everything in coordinates). We start by introducing the vector $b \in T_q\calM$ defined by
\[
b^i = 2 g^{ij}(q) f_j(\eta) = g^{ij}(q)\big(p_j + f_j(q)\big).
\]
Note that $b$ corresponds to the control in the ``controlled gradient flow'' viewpoint described in Section~\ref{section2}, with $b = \dot{q} - \nabla F(q)$. The time-derivative of $b$ along the flow is
\[
D_t b = \nabla^2 F(q) b, 
\]
where $D_t$ denotes the covariant derivative along $q$, while $\nabla^2F$ denotes the Hessian of $F$. We refer to~\cite{Leger2018} for a proof. 

We now recall the result from Lemma~\ref{lemma:manifolds:first-derivative}, which provides the first time-derivative of $G$ along an optimal flow~\eqref{eq:manifolds:HF-q}. Although the lemma is stated in the $(\eta,\eta^*)$ variables, it is easier here to express it in terms of $b$, as follows
\begin{align*}
\frac{dG}{dt} &= g^{ij}(q) f_i(\eta) f_j(\eta) - cH \\
 &= \frac{1}{4} g_{ij}(q) \,b^i \,b^j - cH \\
 &= \frac{1}{4} \bracket{b, b}_q - cH
\end{align*}
where we denote $\bracket{\cdot, \cdot}_q = g(q)(\cdot, \cdot)$. Consequently, the second time-derivative of $G$ is simply written as
\[
\frac{d^2G}{dt^2} = \frac{1}{2}\bracket{D_tb, b}_q = \frac{1}{2} \bracket{\nabla^2F(q)b, b}_q .
\]
Writing this last expression in coordinates, in terms of $\eta$, yields exactly the result. Finally, a similar computation can be used to obtain the expression of $\frac{dG^*}{dt}$. 

\end{proof}

\begin{proof}[Proof of Theorem~\ref{th:manifolds:inequalities}]
The inequality is a consequence of Lemma~\ref{lemma:manifolds:first-derivative} and Lemma~\ref{lemma:manifolds:second-derivative}. Indeed combining the lemmas we obtain that
\[
\frac{d^2}{dt^2}G \ge 2\lambda \left(\frac{d}{dt}G + cH\right)
\]
and
\[
\frac{d^2}{dt^2}G^* \ge -2\lambda \left(\frac{d}{dt}G^* - cH\right)
\]
along the flow~\eqref{eq:manifolds:HF-q}. We can then make use of the following lemma proven in~\cite{Conforti2018_second}:
\begin{lemma*}[Lemma 4.1 in~\cite{Conforti2018_second}]
Let $\phi\colon [0,1]\to\Real$ be twice differentiable on $(0, 1)$ and continuous on $[0,1]$. Let $\lambda\in\Real$. If $\ddot{\phi} +2\lambda\dot{\phi} \ge 0$ on $(0,1)$ then 
\[
\phi_t \le (1-\alpha_t)\,\phi_0 + \alpha_t \,\phi_1,
\]
where $\alpha_t$ is defined in Theorem~\ref{thm:ES} by
\[
\alpha_{t} = \frac{1-e^{-2\lambda t}}{1-e^{-2\lambda}}.
\]
\end{lemma*}

Applying this lemma to the functions $t\to G(\eta_t,\eta^*_t) + cHt$ and $t\to G^*(\eta_t,\eta^*_t) - cHt$ proves the result.
\end{proof}

\begin{example}[Quadratic potential on flat space]
Consider the simple case where $M=\Rn$ is equipped with the flat metric $g_{ij}(q) = \delta_{ij}$. Let 
\[
F(q)=\frac{1}{2} \bracket{A \,q, q},
\]
where $A$ is a symmetric positive definite $n\times n$ matrix and where $\bracket{\cdot,\cdot}$ denotes the canonical inner product on $\Rn$. 
The variational problem~\eqref{eq:gsb-lagrangian} then takes the form
\[
\inf_{q} \int_0^1 \frac{1}{2}\norm{\dot{q}(t)}^2 + \frac{1}{2}\norm{A \,q(t)}^2\,dt
\]
where the infimum runs over all paths $q$ with fixed endpoints, say $q(0)=x$ and $q(1)=y$. Introducing the dual variable $p$, the optimality conditions read
\begin{align*}
\dot{q} &= p, \\
\dot{p} &= A^2\,q .
\end{align*}
The Hopf--Cole transformation $(\eta,\eta^*)\to (q,p)$ is then given by $A\,q = A\,\eta + A\,\eta^*$ and $p = A\,\eta - A\,\eta^*$, i.e. 
\begin{align*}
\eta = \frac{1}{2}\big(q + A^{-1} p \big), \\
\eta^* = \frac{1}{2}\big(q - A^{-1} p \big) .
\end{align*}
We now focus on the splitting of $F$. Let $\lambda > 0$ be the lowest eigenvalue of $A$. It is easy to check that $F$ is $2$-homogeneous, in the sense of Definition~\ref{def:manifolds:Ca}. As a consequence, we can use Theorem~\ref{th:manifolds:inequalities} to split $F(q) = G(\eta,\eta^*) + G^*(\eta,\eta^*)$ with
\[
G(\eta,\eta^*) = \frac{1}{2} \bracket{ A\,\eta,  \eta + \eta^*}
\]
and 
\[
G^*(\eta,\eta^*) = \frac{1}{2} \bracket{A\,\eta^*,  \eta + \eta^*}
\]
The first time-derivatives along the optimal flow are given by
\[
\frac{d}{dt}G(\eta,\eta^*) = \norm{A\,\eta}^2
\]
and 
\[
\frac{d}{dt} G^*(\eta,\eta^*) = -\norm{A\,\eta^*}^2 ,
\]
where $\norm{\cdot}$ denotes the canonical Euclidean norm on $\Rn$. The second time-derivatives are given by
\[
\frac{d^2}{dt^2}G(\eta,\eta^*) = 2\,\bracket{A^2\eta, A\,\eta} \ge 2\,\lambda \norm{A\,\eta}^2
\]
and 
\[
\frac{d^2}{dt^2} G^*(\eta,\eta^*) = -2\,\bracket{A^2\eta^*, A\,\eta^*} \le -2\,\lambda \norm{A\,\eta^*}^2. 
\]
which are the dissipation rates given in the theorem. 

\end{example}

\subsection{Hopf-Cole transformation on graphs}\label{HCgraph}
In this section we consider the generalized Schr{\"o}dinger bridge problem on a discrete probability set. It is a finite-dimensional variational problem whose critical point satisfies a finite-dimensional Hamiltonian flow. We shall derive a (generalized) Hopf--Cole transformations for this Hamiltonian flow, by applying the theory developed in Section~\ref{section5}. 

To start, let us briefly review the $L^2$-Wasserstein metric tensor on graphs \cite{chow2012,maas2011gradient, M}. See related geometry studies in \cite{Li2018_geometrya}. 

Consider a weighted undirected graph $G=(V, E, ,\omega)$, where $V=\{1,\cdots, n\}$ is the vertex set, $E$ is the edge set, and $\omega$ is the weight function defined on the edge set. The probability simplex supported on the vertex set is defined as
\begin{equation*}
\mathcal{P}=\{(\rho_i)_{i=1}^n\in \mathbb{R}^n\colon \sum_{i=1}^n\rho_i=1,~ \rho_i\geq 0 \}.
\end{equation*}
Its interior is denoted by $\mathcal{P}_+$. Denote the tangent space at $\rho\in \mathcal{P}_+$ by
\begin{equation*}
T_\rho\mathcal{P}_+=\Big\{\dot\rho \in \mathbb{R}^n\colon \sum_{i=1}^n\dot\rho_i=0 \Big\}.   
\end{equation*}
Given $\dot\rho^i\in T_\rho\mathcal{P}_+$, $i\in\set{1,2}$, the $L^2$-Wasserstein metric tensor is given by
\begin{equation*}
g_\rho^W(\dot\rho^1, \dot\rho^2)=\dot\rho^{1\ts}L(\rho)^{\dd}\dot\rho^2,    
\end{equation*}
where $\dd$ is the pseudo-inverse operator, $L(\rho)\in \mathbb{R}^{n\times n}$ is the discrete weighted Laplacian matrix  
\begin{equation*}
 L(\rho)_{ij}=\begin{cases}
\sum_{i'=1}^n\omega_{ii'}\theta_{ij}(\rho)& \textrm{if $i=j$}\\
-\omega_{ij}\theta_{ij}(\rho)& \textrm{if $i\neq j$},
     \end{cases}
\end{equation*}
and $\theta_{ij}(\rho)=\frac{1}{2}(\frac{1}{d_i}\rho_i+\frac{1}{d_j}\rho_j)$ with $d_i=\frac{\sum_{i'=1}^n\omega_{ii'}}{\sum_{i=1}^n\sum_{j=1}^n\omega_{ij}}$. We note here that several other choices of functions $\theta_{ij}$ are available in \cite{maas2011gradient}. 

As a consequence, the Riemannian gradient operator of a functional $F\colon \mathcal{P}_+\rightarrow \mathbb{R}$ is given by 
\begin{equation*}
\begin{split}
\textrm{grad}_WF(\rho)=&(L(\rho)^{-1})^{-1}d\mathcal{F}(\rho)\\
=&L(\rho)d\mathcal{F}(\rho)\\
=&-\sum_{j=1}^n\omega_{ij}(\partial_{i}\mathcal{F}(\rho)-\partial_{j}\mathcal{F}(\rho))\theta_{ij}(\rho).
\end{split}
\end{equation*}

The analogue of GSBP \eqref{eq:GSB2} on graph is the following finite-dimensional controlled gradient flow problem:
\begin{equation*}
\inf_{\rho,\sigma}\int_0^1 \sigma(t)^{\ts}L(\rho(t))\sigma(t) \,dt
\end{equation*}
subject to
\begin{equation*}
\dot\rho(t)=\sigma(t)-\textrm{grad}\mathcal{F}(\rho)=\sigma(t)-L(\rho(t))d\mathcal{F}(\rho).    
\end{equation*}
The equivalent, ``Lagrangian-mechanics'' version of this problem is given as follows. Substitute $\sigma(t)=\dot\rho(t)+\textrm{grad}F(\rho)$, then the previous minimization problem is equivalent to
\begin{equation*}
\inf_{\rho}\int_0^1\dot\rho(t)^{\ts}L(\rho(t))\dot\rho(t)+d\mathcal{F}(\rho(t))^{\ts}L(\rho(t))d\mathcal{F}(\rho(t))dt+\mathcal{F}(\rho(1))-\mathcal{F}(\rho(0)).
\end{equation*}
The critical path of the controlled gradient flow problem satisfies Hamilton's equations
\begin{equation*}
\left\{
\begin{aligned}
\dot\rho=&\nabla_S \mathcal{H}(\rho, S)\\
\dot S=&-\nabla_\rho \mathcal{H}(\rho, S),
\end{aligned}\right.
\end{equation*}
where the Hamiltonian is 
\begin{equation*}
\begin{split}
\mathcal{H}(\rho, S)=& \frac{1}{2}S^{\ts}L(\rho)S - \frac{1}{2}d\mathcal{F}(\rho)^{\ts}L(\rho)d\mathcal{F}(\rho).
\end{split}
\end{equation*}
The momentum $S\in T^*_\rho\mathcal{P}_+$ is the covector associated to the velocity $\dot\rho$, i.e. $\dot\rho=L(\rho)S$. Here Hamilton's equations can be formulated explicitly:
\begin{equation}\label{HODE}
\left\{
    \begin{aligned}
&\frac{d\rho_i}{dt}+\sum_{j\in N(i)}\omega_{ij}(S_j-S_i)=0\\
&\frac{dS_i}{dt}+\frac{1}{2}\sum_{j\in N(i)}\omega_{ij}(S_i-S_j)^2\frac{\partial\theta_{ij}}{\partial\rho_i}=\frac{1}{2}\partial_{\rho_i}I(\rho) ,
    \end{aligned}\right.
\end{equation}
where we use the notation $$\mathcal{I}(\rho)=d\mathcal{F}^{\ts}L(\rho)d\mathcal{F}=\frac{1}{2}\sum_{i=1}^n\sum_{j=1}^n\omega_{ij}(\partial_i\mathcal{F}(\rho)-\partial_j\mathcal{F}(\rho))^2\theta_{ij}(\rho).$$ Here $\frac{1}{2}$ is due to the fact that each edge is counted twice in the summation. 

As an analogue to the continuous states case, we call the first equation in \eqref{HODE} the continuity equation on graphs and the second equation in \eqref{HODE} the Hamilton--Jacobi equation on graphs.

We now turn our attention to the generalized Hopf--Cole transformation on graphs. The new variables $\eta$ and $\eta^*$ are given by
\begin{equation}\label{HCG}
\left\{
\begin{aligned}
\partial_i\mathcal{F}(\eta)=&\frac{1}{2}\Big(\partial_i\mathcal{F}(\rho)+S_i\Big)\\
\partial_i\mathcal{F}(\eta^*)=&\frac{1}{2}\Big(\partial_i \mathcal{F}(\rho)-S_i)\Big).
\end{aligned}\right.
\end{equation}
We are now able to present an important example illustrating Hopf--Cole transformation on graphs.

\begin{example}[Discrete entropy]
Let $F(\rho)=\sum_{i=1}^n(\rho_i\log\rho_i-\rho_i)$. We then have $\partial_i F(\rho)=\log\rho$ and the Hopf--Cole transformation on graphs \eqref{HCG} takes the form
\begin{equation*}
\left\{
    \begin{aligned}
\log\eta_i=& \frac{1}{2}\log\rho_i+\frac{1}{2}S_i\\
\log\eta_i^*=&\frac{1}{2}\log\rho_i-\frac{1}{2}S_i     
    \end{aligned}
\right. \implies 
\left\{
\begin{aligned}
\eta_i=&\sqrt{\rho_i}e^{\frac{1}{2}S_i}\\
\eta_i^*=&\sqrt{\rho_i}e^{-\frac{1}{2}S_i} .
\end{aligned}\right.
\end{equation*}
We can therefore rewrite equation \eqref{HODE} in the variables $(\eta, \eta^*)$ and illustrate that we have a symplectic transformation.
\begin{proposition}
We can write
\begin{equation*}
 \left\{
    \begin{aligned}
\frac{d}{dt}\eta_i=&-\frac{1}{2}\partial_{\eta^{*}_i}\mathcal{K}(\eta, \eta^*)\\
\frac{d}{dt}\eta^*_i=&\frac{1}{2}\partial_{\eta_i}\mathcal{K}(\eta, \eta^*),    
    \end{aligned}
        \right.
\end{equation*}
where $\mathcal{K}(\eta,\eta^*)=-2 \eta^{\ts}L(\eta\eta^*)\eta^*$ is the Hamiltonian in the new variables. 
\end{proposition}
\begin{proof}
Notice that $\textrm{Hess} \mathcal{F}(\rho)=\textrm{diag}(\frac{1}{\rho})$, thus 
\begin{equation*}
\begin{split}
    \sigma(\eta, \eta^*)=&\frac{1}{2}\textrm{Hess~} \mathcal{F}(\eta)^{-1}\textrm{Hess~} \mathcal{F}(\rho) \textrm{Hess~} \mathcal{F}(\eta^*)^{-1}\\
    =&\frac{1}{2}\textrm{diag}(\eta) \textrm{diag}(\frac{1}{\rho}) \textrm{diag}(\eta^*)\\
    =&\frac{1}{2}\textrm{diag}(\frac{\eta\eta^*}{\rho})=\frac{1}{2}\mathbb{I},
\end{split}
\end{equation*}
where $\mathbb{I}$ is the identity matrix. From Lemma \ref{lemma:manifolds:HF-eta}, we derive the equation for $\eta$, $\eta^*$.
\end{proof}
We now wish to rewrite the equation of $\eta$ explicitly. Similarly equation can be written for $\eta^*$.
\begin{equation*}
\frac{d}{dt}\eta_i=-\frac{1}{2}\sum_{j\in N(i)}\omega_{ij}(\eta_i-\eta_j)\theta_{ij}(\eta\eta^*)-\frac{1}{2}\sum_{j\in N(i)}\omega_{ij}(\eta_i-\eta_j)(\eta_i^*-\eta_j^*)\partial_{\eta_i^*}\theta_{ij}(\eta\eta^*).
\end{equation*}
The continuous analogue of the above formula gives 
\begin{equation*}
\begin{split}
\partial_t\eta=&\frac{1}{2}\nabla\cdot(\eta\eta^* \nabla \eta)-\frac{1}{2}(\nabla\eta,\nabla\eta^*)\eta.
\end{split}
\end{equation*}
We compare it to the one in continuous state, i.e the Schrödinger system \eqref{bfheat}. From this comparison, we discover the following relation: 

\noindent\textbf{Claim:}
\begin{equation*}
    -\Delta\eta=\frac{1}{2}\nabla\cdot(\eta\eta^* \nabla \eta)-\frac{1}{2}(\nabla\eta,\nabla\eta^*)\eta.
\end{equation*}
\begin{proof}[Proof of Claim:]
Notice the fact that in the continuous state, 
\begin{equation*}
\mathcal{K}(\eta,\eta^*)=\int (\nabla\eta,\nabla\eta^*) dx=\int (\eta\nabla\log\eta, \eta^*\nabla\log\eta^*)dx. 
\end{equation*}
There are two variation formulation of $\delta_{\eta^*}\mathcal{K}$, which proves the claim. 
\end{proof}
Here the claim presents a nonlinear reformulation of Laplacian operator. It is hidden in Hopf--Cole transformation. This fact shares many similarities to the nonlinear reformulation of Laplacian operator derived in Schr{\"o}dinger equation \cite{ChowLiZhou2017_discrete}. It is one of the motivation for this paper. 
\end{example}

\section{Discussion}
In this work, we study the generalized SBP as a controlled gradient flow problem in the Wasserstein space (density manifold). We discuss the symplectic structures of (generalized) Hopf--Cole transformation. Similar structures can also be extended to general homogeneous metrics with various boundary conditions in density manifold. In the future, we will research related problems on statistical manifold~\cite{IG2}, with applications in both mean field games and machine learning. 
\bibliographystyle{abbrv}
\bibliography{hpc}

\section*{Notations}
We will use the following notations

\begin{tabular}{|l|c|}
\hline 
Base manifold & $M$\\
Drift vector field & $b$\\
Divergence operator & $\textrm{div}$\\
Gradient operator & $\nabla$\\
\hline
Density manifold & $\mathcal{P}_+(M)$ \\
Probability distribution & $\rho$ \\
Tangent space & $\dot\rho\in T_\rho\mathcal{P}_+(\Omega)$ \\
Wasserstein metric tensor & $g^W_\rho$\\
Weighted Laplacian operator & $\Delta_\rho=\nabla\cdot(\rho\nabla\cdot)$\\
Dual coordinates & $S$\\
First $L^2$ variation & $\delta$\\
 Second $L^2$ variation &$\delta^2$ \\ 
Gradient operator & $\textrm{grad}_W$\\
Hessian operator &$\textrm{Hess}_W$ \\
Christoffel symbol & $\Gamma_\rho^W(\cdot, \cdot)$\\
Cotangent space & $(\rho, \Phi)\in T_\rho^*\mathcal{P}(M)$\\
Hopf--Cole variables & $(\eta, \eta^*)\in\mathcal{C}(M)$\\
Tangent space for Hopf--Cole variables & $(\dot\eta, \dot\eta^*)\in\mathcal{D}(M)$\\
Density manifold symplectic form & $\Omega_{\rho}^W(\cdot, \cdot)$\\
Hopf--Cole symplectic form & $\Omega_\rho^\mathcal{D}(\cdot,\cdot)$\\
Split energies & $\mathcal{G}$, $\mathcal{G}^*$\\
\hline
\end{tabular}

\begin{tabular}{|l|c|}
\hline 
Finite dimensional manifold &
$q\in\mathcal{M}$ \\
Metric tensor & $g$\\
Energy & $F$\\
Split energies & $G$, $G^*$\\
\hline
\end{tabular}

\end{document}